%% file: ex_article.tex
\documentclass[review]{siamart1116}

\input{ex_shared}

\ifpdf
\hypersetup{
  pdftitle={\TheTitle},
  pdfauthor={\TheAuthors}
}
\fi

\begin{document}

\maketitle

\begin{abstract}
  This paper presents a modified quasi-reversibility method for computing the exponentially unstable solution of a nonlocal terminal-boundary value parabolic problem with noisy data. Based on data measurements, we perturb the problem by the so-called filter regularized operator to design an approximate problem.  Different from recently developed approaches that consist in the conventional spectral methods, we analyze this new approximation in a variational framework, where the finite element method can be applied. To see the whole skeleton of this method, our main results lie in the analysis of a semi-linear case and we discuss some generalizations where this analysis can be adapted. As is omnipresent in many physical processes, there is likely a myriad of models derived from this simpler case, such as source localization problems for brain tumors and heat conduction problems with nonlinear sinks in nuclear science.
  With respect to each noise level, we benefit from the Faedo-Galerkin method to study the weak solvability of the approximate problem. Relying on the energy-like analysis, we provide detailed convergence rates in $L^2$-$H^1$ of the proposed method when the true solution is sufficiently smooth. Depending on the dimensions of the domain, we obtain an error estimate in $L^{r}$ for some $r>2$. Proof of the backward uniqueness for the quasi-linear system is also depicted in this work. To prove the regularity assumptions acceptable, several physical applications are discussed.
\end{abstract}



\section{Introduction}
\subsection{Background of the terminal value model}\label{subsec:background}
This paper is concerned with a general construction of a modified quasi-reversibility method for a  quasi-linear parabolic reaction-diffusion system of the following form
\begin{equation}
u_{t}+\nabla\cdot\left(-a\left(x,t;u;\nabla u\right)\nabla u\right)=F\left(x,t;u;\nabla u\right)\quad\text{in }Q_{T}:=\Omega\times\left(0,T\right),
\label{Pb1}
\end{equation}
where the vector of concentrations $u = u(x,t)\in \mathbb{R}^{N}$ is unknown with $N\in \mathbb{N}^{*}$ being the number of equations involved in \eqref{Pb1}.
Here, the domain of interest $\Omega \subset \mathbb{R}^{d}$ for $d\in \mathbb{N}^{*}$ and the final time of observation $0<T<\infty$ are assumed. Furthermore, $\Omega$ is open, connected and bounded with a sufficiently smooth boundary $\partial\Omega$. The nonlocal diffusion coefficient $a \in \mathbb{R}^{N\times N}$ and the nonlinearity $F \in \mathbb{R}^{N}$ are explicitly density- and gradient-dependent.

As met in practical applications, we associate \eqref{Pb1} either with the homogeneous Dirichlet boundary condition ($u = 0$ on $\partial\Omega$) or with the homogeneous Neumann boundary condition ($-a(x,t;u;\nabla u)\nabla u\cdot \text{n} = 0$ on $\partial\Omega$). Given the terminal data
\begin{equation}
	u\left(x,T\right)=u_{f}\left(x\right)\quad\text{in }\Omega,
	\label{terminal}
\end{equation}
we would like to seek in this work the initial value $u\left(x,0\right)=u_{0}\left(x\right)$ in a stable way since the solution to this type of problems is highly unstable (cf. e.g. \cite{Kaba08}).

The motivation behind the consideration of \eqref{Pb1}-\eqref{terminal} basically follows the identification of source location for brain tumor that has been investigated in \cite{JBAJ16}. It is worth mentioning that reconstructing the initial densities of tumor cells provides a substantial contribution to predicting tumorigenicity in connection with genetic events (see this possible correlation studied in e.g. \cite{Zlat01}). The spirit of studying the terminal data \eqref{terminal} also arises in the theory of Kolmogorov backward equations, carried out by nonlocal transformations in \cite{BS04} e.g., to  integrate the expected value of the payout from future values. Therefore, the problem under consideration here is viewed as a prototypical framework which can be adapted to particular applicable contexts and be extended to other theoretical approaches.

The existing literature on quasi-linear reaction-diffusion systems is very huge to be singled out here. Since the diffusion tensor $a$ in \eqref{Pb1} is nonlinear with included self- and/or cross diffusion types, there are of course numerous distinctive aspects concerning different types of forward problems considered here. For example, discussions on well-posedness, spectrum analysis and behaviors of travelling waves have been detailed in \cite{MRS14} and references cited therein, see e.g. \cite{SS01,Mey11,CD00}. We also wish to mention here the works \cite{CGT17,Wei92,HR08} for addressing more complex scenarios related to either theoretical or numerical standpoints of reaction-diffusion type systems.

\subsection{Goals and novelty}
The purpose here is to follow up on our earlier work \cite{TAKL17}, where we have proposed a regularization strategy in the vein of the classical \emph{quasi-reversibility} (QR) method which specifically solves ill-posed problems of elliptic and parabolic types. Observe that the identification of population density for a single-species model in \cite{TAKL17} is well-suited to the concept of source localization. In this sense, behaviors of the tumor cell density are influenced not only by certain proliferation and/or extinction rates, but also by their transport processes with convection/advection, and the total population in local movements. The novelty we present here is the careful adaptation of the \emph{filter regularized operator}, which we have briefly studied in \cite{DNKV17}, to the modified QR method in a variational framework. Remarkably, this setting enables us to interact with certain reaction-diffusion problems with spatial nonlinear diffusion; compared to our spectral-based regularization methods (cf. e.g. \cite{TAKL17,TQ11,2010}) that have been studied so far.

The motivation for using the QR type method stems from our wish to design a regularization approach that can deal with a quite general class of parabolic problems due to the limitation of regularization theory. As is known, regularization of many simpler models has been deduced so far, such as the heat conduction problem (e.g. \cite{TDMK15,RE97} and references therein), the parabolic problems in image restoration (cf. \cite{Cara13,Cara76,Cara78,Cara94}), the Burger equation in fluid mechanics \cite{Cara77} and even the Navier-Stokes equation \cite{KP68}. However, it is impossible to find papers working on time-reversed quasi-linear systems \eqref{Pb1}-\eqref{terminal} except our previous work \cite{TAKL17} that has been mentioned above. Our major contribution here is thus coping with problems that remain unsolved until now. We stress that \eqref{Pb1} includes not only popular semi-linear types (e.g. equations/systems named Fitzhugh-Nagumo, Fisher-KPP, Zeldovich, Lengyel-Epstein, de Pillis-Radunskaya and  Frank-Kamenetsky; see \cite{MOEP16,EK14,DAAF16,CQ04,BE89,Pao92}
for the background of deterministic models), but also certain nonlocal types in e.g. \cite{RS92,Lan16,CKK16,KLW17}. In principle, our mathematical results derived here are helpful in fostering interests in the branch of inverse and ill-posed problems for partial differential equations. Alternative approaches to design a regularized
problem can be the quasi-boundary value method (commenced in \cite{Show83}
and numerically discussed in \cite{KTDT17} e.g.), the truncation
method (see, e.g. \cite{2010,RR06}) and the recently developed Tikhonov method based on Carleman weight functions (cf. \cite{Kli15}). In addition, backward problems with
impulsive and random noise have been investigated in \cite{KWH16} by the generalized
Tikhonov regularization and in \cite{KNT17} by a QR-based statistical approach, respectively.

We accentuate that this work is not aimed to improve the conventional convergence rates of this method, but to complete the theoretical error analysis of this direction. Together with rigorous $L^2$-$H^1$ error estimates, we obtain an $L^{r}$-type rate ($r>2$) of convergence, which we believe that this is the first time it is explored in this direction. Theoretically, our work also unravels the problem of finding the global-in-time error estimate. To be more specific, we recall the analysis in \cite{Klibanov2015}, where a linear case of \eqref{Pb1} was considered through a version of the QR method. Proofs of the stability and error estimates in \cite{Klibanov2015} are concretely based on the massive Carleman estimate, but it is well known that this method often requires $T$ sufficiently small; see \cite[Theorem 5.1]{Klibanov2015}. This price is also manifested here when we prove the backward uniqueness result for \eqref{Pb1} using a Carleman-type estimate with a suitable non-increasing weight function; see \cref{lem:3.1}. Here, the rates of convergence we obtain for the semi-linear case are similar to \cite[Theorem 5.4]{Klibanov2015}, but are uniform in time, requiring a very high smoothness of the true solution somewhat in terms of Gevrey spaces. To prove this, an exponentially decreasing weight function is used to get rid of large parameters appearing in the difference problem. Essentially, this largeness is driven by the magnitude stability of the regularized problem, which goes to infinity when the measurement parameter tends to zero; see e.g. \cite{Nam10} for various types of the magnitude in the past. In accordance with the existence result of the regularized problem, this way the proofs of convergence would be simpler using a large amount of energy-like estimates.

\subsection{Contemporary history of the QR method}
The QR method was first proposed by Latt\`es and Lions in the monograph \cite{LL67}. This method, when applied to the context of linear backward parabolic problems, basically perturbs the spatial second-order operator by the addition of a fourth-order term. It is, on the other hand, going with a leading parameter which is positive and small enough to get the convergence. Additionally, the sign of this extra term is chosen such that the perturbed/regularized problem is well-posed with respect to the leading parameter, as time evolves back to the initial point. In the community of regularization, this parameter is referred to as the \emph{regularization parameter}. Let us also note that since our work aims to prepare the playground to handle real-world models, the presence of noise on the terminal data is evident. Accordingly, the smallness of the regularization parameter here depends strictly on such noise levels, which makes our scheme applicable in reality.

Having massive research interests for five decades, the literature of the QR method and its modifications (e.g. the stabilized QR method in \cite{Mill67}), nowadays, is vast from the vantage point of theoretical and numerical analysis. As some of concrete references for elliptic equations, a dual-based QR method for the Cauchy problem with noisy data is designed
in \cite{BD10} and some numerical approaches have been postulated in \cite{CKP09}. On the other hand, the error analysis is very attractive and has been investigated, for example, in \cite{KS91} with
the H\"older-type rate and in \cite{BD09} with a logarithmic rate. This method is also extended to deal with inverse problems for parabolic and hyperbolic equations; see e.g. \cite{CK07,Kliba05} for a brief overview of this field with sharp error estimates and convergence results in $H^1$. On top of that, the reader can be referred to \cite{Klibanov2015} as a survey of applications of Carleman estimates to proofs of convergence of the QR method for a wide class of ill-posed problems for PDEs.

\subsection{Outline of the paper}
From a mathematical point of view, the nonlinearities $a$ and $F$ involved in \eqref{Pb1} are undoubtedly the major challenges. Here we aim at showing the general setting of the QR method and thereupon explaining the ideas on a simpler case while leaving the more general case of \eqref{Pb1} to future works in this inception stage. For this reason, we introduce in \eqref{reguproblem} the regularized problem for the general system \cref{Pb1}, while we reduce ourselves to the analysis of a semi-linear case with single-species mode. In this regard, we will not also detail any further technical assumption of the diffusion tensor $a$, except the ellipticity condition in the general form of \eqref{Pb1} that serves the convergence analysis. Notice that proofs of our main results are done with the zero Dirichlet boundary condition, which plays a key role in the variational framework we choose. As shortly discussed in the last part of \cref{sec:remarks} these results are also obtainable for the zero Neumann boundary condition.

Except the notation and necessary assumptions on the input of the problem in the general form are present in \cref{sec:2}, our main themes in this paper can be summarized as follows:
\begin{itemize}
	\item  Detailed settings of the modified QR method and weak solvability of the regularized problem are studied in \cref{sec:4}; see \cref{existencetheorem} and \cref{posi}. 
	\item Detailed rates of convergence are obtained in \cref{sec:5}, where the main results are reported in \cref{thm:Estimate1}, \cref{thm:Estimate2} and \cref{cor:5.3}, respectively.
	\item Some particular extensions follow in \cref{sec:remarks}, including the uniqueness result for the system \eqref{Pb1}-\eqref{terminal}.
\end{itemize}
Finally, some working applications are present in \cref{appendix1}.


\section{Preliminaries}
\label{sec:2} 

Let $\left\langle \cdot,\cdot\right\rangle $ be either the scalar
product in $L^{2}$ or the dual pairing of a continuous linear functional
and an element of a function space. The notation $\left\Vert \cdot\right\Vert _{X}$
stands for the norm in the Banach space $X$. We call $X'$ the dual
space of $X$. We denote by $L^{p}\left(0,T;X\right)$, $1\le p\le\infty$
for the Banach space of real-valued functions $u:\left(0,T\right)\to X$
measurable, provided that
\[
\left\Vert u\right\Vert _{L^{p}\left(0,T;X\right)}=\left(\int_{0}^{T}\left\Vert u\left(t\right)\right\Vert _{X}^{p}dt\right)^{1/p}<\infty\quad\text{for }1\le p<\infty,
\]
while
\[
\left\Vert u\right\Vert _{L^{\infty}\left(0,T;X\right)}=\underset{t\in\left(0,T\right)}{\text{ess sup}}\left\Vert u\left(t\right)\right\Vert _{X}<\infty\quad\text{for }p=\infty.
\]
We denote the norm of the function space $C^{k}\left(\left[0,T\right];X\right)$,
$0\le k\le\infty$ by
\[
\left\Vert u\right\Vert _{C^{k}\left(\left[0,T\right];X\right)}=\sum_{n=0}^{k}\sup_{t\in\left[0,T\right]}\left\Vert u^{\left(n\right)}\left(t\right)\right\Vert _{X}<\infty.
\]
We denote by $H_{0}^{1}\left(\Omega\right)$ for the Hilbert space
of weakly differentiable functions $u:\Omega\to\mathbb{R}$ that vanishes
on the boundary in the sense of trace. On the other hand, $W^{p,q}\left(\Omega\right)$
for $p\in\mathbb{N}$ denotes the Sobolev space of functions with
index of differentiability $p$ and of integrability $q$ (if $q\in\mathbb{N}$)
or, in the case $q=\infty$, whose essential supremum exists.

Depending on the situation, we  denote by $\left|\cdot\right|$ either the absolute value of a function or the finite-dimensional Euclidean norm of a vector. There are several assumptions needed for the analysis below:

$\left(\text{A}_{1}\right)$ The diffusion tensor $a = \left(a_{ij}\right)_{1\le i,j\le N}$ is such that the mapping $\left(\mathbf{p},\mathbf{q}\right)\mapsto a\left(x,t;\mathbf{p};\mathbf{q}\right)$ is continuous for $\left(\mathbf{p},\mathbf{q}\right)\in [L^{2}\left(\Omega\right)]^{N}\times\left[L^{2}\left(\Omega\right)\right]^{Nd}$ and the mapping  $(x,t)\mapsto a\left(x,t;\mathbf{p};\mathbf{q}\right)$ is continuously differentiable for $(x,t)\in \overline{Q_{T}}$ . Moreover, there exists a positive constant $\overline{M}$ such that
\[
0<\sum_{i,j=1}^{N}a_{ij}\left(x,t;\mathbf{p};\mathbf{q}\right)\xi_{i}\xi_{j}\le\overline{M}\left|\xi\right|^{2}\,\text{for all }\xi\in\mathbb{R}^{N},\left(\mathbf{p},\mathbf{q}\right)\in [L^{2}\left(\Omega\right)]^{N}\times\left[L^{2}\left(\Omega\right)\right]^{Nd}.
\]

$\left(\text{A}_{2}\right)$ There exists a tensor $A(x,t;\mathbf{p},\mathbf{q})\in \mathbb{R}^{N\times N}$ such that $A_{ij} = \overline{M} - a_{ij}$ for $1\le i,j\le N$. Then there exists a positive constant $\underline{M}$ satisfying
\[
0<\underline{M}\left|\xi\right|^{2} \le \sum_{i,j=1}^{N}A_{ij}\left(x,t;\mathbf{p};\mathbf{q}\right)\xi_{i}\xi_{j}\le\overline{M}\left|\xi\right|^{2},
\]
for all $\xi\in\mathbb{R}^{N},\left(\mathbf{p},\mathbf{q}\right)\in [L^{2}\left(\Omega\right)]^{N}\times\left[L^{2}\left(\Omega\right)\right]^{Nd}$.

$\left(\text{A}_{3}\right)$ For any $\left(x,t\right)\in \overline{Q_T}$, the source function $F$ is measurable and locally Lipschitz-continuous in the sense that for $1\le i\le N$
\[
\left|F\left(x,t;\mathbf{p};\mathbf{q}\right)-F\left(x,t;\mathbf{r};\mathbf{s}\right)\right|\le L_{F}\left(\ell\right)\left(\left|\mathbf{p}-\mathbf{r}\right|+\left|\mathbf{q}-\mathbf{s}\right|\right),
\]
for $\max\left\{\left|\mathbf{p}\right|,\left|\mathbf{r}\right|,\left|\mathbf{q}\right|,\left|\mathbf{s}\right|\right\}\le \ell$  for some $\ell >0$.

$\left(\text{A}_{4}\right)$ There exists a measurement of $u_{f}$, denoted by $u_{f}^{\varepsilon}$, in $[L^{2}\left(\Omega\right)]^{N}$ such that
\[
\left\Vert u_{f}-u_{f}^{\varepsilon}\right\Vert _{[L^{2}\left(\Omega\right)]^{N}}\le\varepsilon,
\]
where $\varepsilon>0$ represents the noise level.

\begin{remark}\label{remarkphu}
	It follows from $\left(\text{A}_{3}\right)$ that we can take
	\begin{align*}
	L_{F}\left(\ell\right) & :=\sup\left\{ \frac{\left|F\left(x,t;\mathbf{p};\mathbf{q}\right)-F\left(x,t;\mathbf{r};\mathbf{s}\right)\right|}{\left|\mathbf{p}-\mathbf{r}\right|+\left|\mathbf{q}-\mathbf{s}\right|}\right.\\
	& :\left(x,t\right)\in Q_{T},\mathbf{p}\ne \mathbf{r},\mathbf{q}\ne \mathbf{s}\;\text{and }\left|\mathbf{p}\right|,\left|\mathbf{r}\right|,\left|\mathbf{q}\right|,\left|\mathbf{s}\right|\le\ell\bigg\} <\infty,
	\end{align*}
	for $\ell >0$. Moreover, we introduce the cut-off function $F_{\ell}$, as follows:
	\begin{equation}
	F_{\ell}\left(x,t;\mathbf{p};\mathbf{q}\right):=\begin{cases}
	F\left(x,t;\ell;\ell\right) & \text{if }\max_{1\le j\le N,1\le k\le d}\left\{ \mathbf{p}_{j},\mathbf{q}_{jk}\right\} \in\left(\ell,\infty\right),\\
	F\left(x,t;\mathbf{p};\mathbf{q}\right) & \text{if }\max_{1\le j\le N,1\le k\le d}\left\{ \mathbf{p}_{j},\mathbf{q}_{jk}\right\} \in\left[-\ell,\ell\right],\\
	F\left(x,t;-\ell;-\ell\right) & \text{if }\max_{1\le j\le N,1\le k\le d}\left\{ \mathbf{p}_{j},\mathbf{q}_{jk}\right\} \in\left(-\infty,-\ell\right).
	\end{cases}\label{eq:2.1-1}
	\end{equation}
	
	Due to the cut-off function, for any $\ell>0$
	it holds
	\begin{equation}
	\left|F_{\ell}\left(x,t;\mathbf{p};\mathbf{q}\right)-F_{\ell}\left(x,t;\mathbf{r};\mathbf{s}\right)\right|\le L_{F}\left(\ell\right)\left(\left|\mathbf{p}-\mathbf{r}\right|+\left|\mathbf{q}-\mathbf{s}\right|\right),\label{eq:2.1}
	\end{equation}
	for all $\left(x,t\right)\in Q_{T}$ and $\mathbf{p},\mathbf{r}\in [L^2(\Omega)]^{N},\mathbf{q},\mathbf{s}\in [L^2(\Omega)]^{Nd}$.
	
	The proof of (\ref{eq:2.1}) is trivial. For brevity, we sketch out the proof in the following cases and omit the details:
	
	\underline{Case 1:} $\max_{1\le j\le N,1\le k\le d}\left\{ \mathbf{p}_{j},\mathbf{q}_{jk}\right\} <-\ell$ and $\max_{1\le j\le N,1\le k\le d}\left\{ \mathbf{r}_{j},\mathbf{s}_{jk}\right\} <-\ell;$
	
	\underline{Case 2:} $\max_{1\le j\le N,1\le k\le d}\left\{ \mathbf{p}_{j},\mathbf{q}_{jk}\right\} <-\ell\le\max_{1\le j\le N,1\le k\le d}\left\{ \mathbf{r}_{j},\mathbf{s}_{jk}\right\} \le\ell;$
	
	\underline{Case 3:} $\max_{1\le j\le N,1\le k\le d}\left\{ \mathbf{p}_{j},\mathbf{q}_{jk}\right\} <-\ell<\ell\le\max_{1\le j\le N,1\le k\le d}\left\{ \mathbf{r}_{j},\mathbf{s}_{jk}\right\} ;$
	
	\underline{Case 4:} $-\ell<\max_{1\le j\le N,1\le k\le d}\left\{ \mathbf{p}_{j},\mathbf{q}_{jk}\right\} ,\max_{1\le j\le N,1\le k\le d}\left\{ \mathbf{r}_{j},\mathbf{s}_{jk}\right\} \le\ell;$
	
	\underline{Case 5:} $\max_{1\le j\le N,1\le k\le d}\left\{ \mathbf{p}_{j},\mathbf{q}_{jk}\right\} >\ell$ and $\max_{1\le j\le N,1\le k\le d}\left\{ \mathbf{r}_{j},\mathbf{s}_{jk}\right\} >\ell.$
\end{remark}




\section{General frameworks for the QR method}
\label{sec:4}


This is the moment to establish a regularized problem for the system \eqref{Pb1}-\eqref{terminal} with measured data $u_{f}^{\varepsilon}$. For $\varepsilon > 0$, we denote by $\beta:=\beta\left(\varepsilon\right)\in\left(0,1\right)$ the regularization parameter satisfying
\[
\lim_{\varepsilon\to0^{+}}\beta\left(\varepsilon\right)=0,
\]
and then consider the function $\gamma:\left[0,T\right]\times\left(0,1\right)$ such that for any $\beta>0$, there holds
\[
\quad\gamma\left(T,\beta\right)\ge 1,\quad\lim_{\beta\to0^{+}}\gamma\left(t,\beta\right)=\infty\quad\text{for all }t\in (0,T].
\]

Compared to \cite{DNKV17}, we do not require the fundamental multiplicative-like identities with respect to the first argument in $\gamma$. With the function $\gamma$ at hands, we define the following operators.
\begin{definition}[Perturbing operator]
	The linear mapping $\mathbf{Q}_{\varepsilon}^{\beta}:[L^{2}\left(\Omega\right)]^{N}\to [L^{2}\left(\Omega\right)]^{N}$ is said to be a perturbing operator if there exist a function space $\mathbb{W}\subset [L^{2}\left(\Omega\right)]^{N}$ and an $\varepsilon$-independent constant $C_0 > 0$ such that
	\begin{equation}
	\left\Vert \mathbf{Q}_{\varepsilon}^{\beta}u\right\Vert _{[L^{2}\left(\Omega\right)]^{N}}\le C_{0}\left\Vert u\right\Vert _{\mathbb{W}}/\gamma\left(T,\beta\right)\quad\text{for any }u\in\mathbb{W}.\label{eq:4.1-3}
	\end{equation}
\end{definition}

\begin{definition}[Stabilized operator]
	The linear mapping $\mathbf{P}_{\varepsilon}^{\beta}:[L^{2}\left(\Omega\right)]^{N}\to [L^{2}\left(\Omega\right)]^{N}$ is said to be a stabilized operator if there exists an $\varepsilon$-independent constant $C_1 > 0$ such that
	\begin{equation}
	\left\Vert \mathbf{P}_{\varepsilon}^{\beta}u\right\Vert _{[L^{2}\left(\Omega\right)]^{N}}\le C_{1}\log\left(\gamma\left(T,\beta\right)\right)\left\Vert u\right\Vert _{[L^{2}\left(\Omega\right)]^{N}}\quad\text{for any }u\in [L^{2}\left(\Omega\right)]^{N}.\label{eq:4.1-1}
	\end{equation}
\end{definition}


In principle, the way we define these two terms $\mathbf{P}_{\varepsilon}^{\beta}$ and $\mathbf{Q}_{\varepsilon}^{\beta}$ is in line with the classical quasi-reversibility method. In this sense, we obtain the regularized problem by adding the perturbing operator $\mathbf{Q}_{\varepsilon}^{\beta}$ to the original problem. Then the stabilized operator will be derived from this addition by a linear mapping, whenever the leading coefficients of operators, which are targeted to be stabilized, are essentially bounded. Hence, in this work we simply take $\mathbf{P}_{\varepsilon}^{\beta} := \overline{M}\Delta + \mathbf{Q}_{\varepsilon}^{\beta}$. Interestingly, this enables us to consider a very simple eigenvalue problem regardless of the complex structure involved in the diffusion coefficient.

At the moment, we do not know the optimal bounds of \eqref{eq:4.1-3} and \eqref{eq:4.1-1}, which are altogether related. We deliberately present the logarithmic stability estimate \eqref{eq:4.1-1} for the stabilized operator due to the typical logarithmic convergence usually obtained after the regularization of a backward parabolic model. In other words, this upper bound is essential and decisive in the convergence analysis in \cref{sec:5}. The decay behavior of  the  perturbing operator (cf. \eqref{eq:4.1-3}) is directly governed by the so-called source condition  that measures the high smoothness of the true solution. In the following example, we mimic the stochastic gradient descent algorithm in machine learning schemes to show the existence of these operators.

\begin{example}\label{exampleope}
	Consider $N=1$ for a single-species model. It is well known that for any bounded subset of $\mathbb{R}^{d}$ with a smooth boundary, there exists
	an orthonormal basis of $L^{2}\left(\Omega\right)$, denoted by $\left\{ \phi_{p}\right\} _{p\in\mathbb{N}}$,
	satisfying $\phi_{p}\in H_{0}^{1}\left(\Omega\right)\cap C^{\infty}\left(\overline{\Omega}\right)$
	and $-\Delta\phi_{p}\left(x\right)=\mu_{p}\phi_{p}\left(x\right)$ for $x\in \Omega$. The (Dirichlet and Neumann) eigenvalues $\left\{ \mu_{p}\right\} _{p\in\mathbb{N}}$ form an infinite sequence which goes to infinity, viz.
	\begin{equation}
	0\le \mu_{0}<\mu_{1}\le\mu_{2}\le...,\;\text{and }\lim_{p\to\infty}\mu_{p}=\infty.
	\label{prop}
	\end{equation}
	
	We choose
	\begin{equation}
	\mathbf{Q}_{\varepsilon}^{\beta}u=\frac{1}{T}\sum_{p\in \mathbb{N}}\log\left(1+\gamma^{-1}\left(T,\beta\right)e^{\overline{M}T\mu_{p}}\right)\left\langle u,\phi_{p}\right\rangle \phi_{p}\quad\text{for }u\in L^{2}\left(\Omega\right).\label{eq:exa}
	\end{equation}
	Using the elementary inequality $\log\left(1+a\right)\le a$
	for $a>0$, then by Parseval's identity it gives
	\begin{align*}
	\left\Vert \mathbf{Q}_{\varepsilon}^{\beta}u\right\Vert _{L^{2}\left(\Omega\right)}^{2} & =\frac{1}{T^{2}}\sum_{p\in \mathbb{N}}\log^{2}\left(1+\gamma^{-1}\left(T,\beta\right)e^{\overline{M}T\mu_{p}}\right)\left|\left\langle u,\phi_{p}\right\rangle \right|^{2}\\
	& \le\frac{\gamma^{-2}\left(T,\beta\right)}{T^{2}}\left\Vert e^{\overline{M}T\left(-\Delta\right)}u\right\Vert _{L^{2}\left(\Omega\right)}^{2}.
	\end{align*}
	
	The norm $\left\Vert e^{\overline{M}T\left(-\Delta\right)}u\right\Vert _{L^{2}\left(\Omega\right)}$ is characterized by the so-called Gevrey class of real-analytic functions. In this case, it is also performed as a Hilbert space and then contained in $L^2(\Omega)$. Fruitful discussions on this typical space are preferably in \cref{sec:conclusions} and \cref{appendix1}. It now remains to deduce the estimate for the operator $\mathbf{P}_{\varepsilon}^{\beta}$. In fact, it follows from its own structure that
	\begin{align*}
	\mathbf{P}_{\varepsilon}^{\beta}u  =\sum_{p\in\mathbb{N}}\left(\frac{1}{T}\log\left(1+\gamma^{-1}\left(T,\beta\right)e^{\overline{M}T\mu_{p}}\right)-\overline{M}\mu_{p}\right)\left\langle u,\phi_{p}\right\rangle \phi_{p}.
	\end{align*}
	
	Thanks to the inequality $\log\left(1+ab\right)\le\log\left(b\left(1+a\right)\right)\le\log\left(1+a\right)+\log\left(b\right)$ for $a>0, b\ge 1$, we have
	\[
	\log\left(1+\gamma^{-1}\left(T,\beta\right)e^{\overline{M}T\mu_{p}}\right)-\overline{M}T\mu_{p}\le\log\left(1+\gamma^{-1}\left(T,\beta\right)\right)\quad\text{for all }p\in\mathbb{N}.
	\]
	Consequently, by Parseval's identity we get
	\[
	\left\Vert \mathbf{P}_{\varepsilon}^{\beta}u\right\Vert _{L^{2}\left(\Omega\right)}\le\frac{1}{T}\log\left(\gamma\left(T,\beta\right)\right)\left\Vert u\right\Vert _{L^{2}\left(\Omega\right)}.
	\]
\end{example}


Now, we detail the regularized problem: For each $\varepsilon>0$,
let $\ell^{\varepsilon}:=\ell\left(\varepsilon\right)\in\left(0,\infty\right)$
be a cut-off parameter satisfying
\begin{equation}
\lim_{\varepsilon\to0^{+}}\ell^{\varepsilon}=\infty,
\label{cutoff-para}
\end{equation}
then we consider the following problem:
\begin{equation}
u_{t}^{\varepsilon}+\nabla\cdot\left(-a\left(x,t;u^{\varepsilon};\nabla u^{\varepsilon}\right)\nabla u^{\varepsilon}\right)-\mathbf{Q}_{\varepsilon}^{\beta}u^{\varepsilon}=F_{\ell^{\varepsilon}}\left(x,t;u^{\varepsilon};\nabla u^{\varepsilon}\right)\quad\text{in }Q_{T},
\label{reguproblem}
\end{equation}
associated with the Dirichlet boundary condition and the terminal
noisy data
\begin{equation}
u^{\varepsilon}=0\;\text{ on }\partial\Omega\times\left(0,T\right),\quad u^{\varepsilon}\left(x,T\right)=u_{f}^{\varepsilon}\left(x\right)\;\text{ in }\Omega.
\label{reguproblem-1}
\end{equation}

\section{Analysis of the QR method}
Some certain cases of the general system \eqref{Pb1} can be solved by the QR scheme we have proposed. Nevertheless, this mathematical over-generality merely leads to extra steps of proofs. Thereby, this curtails the core idea behind the regularization. In this section we only consider a rather simplified version of \eqref{Pb1}, while we will briefly discuss the result of the general system in \cref{sec:remarks}. We take into account the following reaction-diffusion equation with $N=1$:\begin{equation}
	u_{t} - a \Delta u=F\left(x,t;u\right)\quad\text{in }Q_{T},
	\label{Pb1-new}
	\end{equation}
endowed with the zero Dirichlet boundary condition and the terminal condition \eqref{terminal}.
	
This means we will use the assumptions  $\left(\text{A}_{1}\right)$-$\left(\text{A}_{4}\right)$ with $N=1$. Notice that \eqref{Pb1-new} also implies the following reduction through our analysis:
\begin{itemize}
	\item $a(x,t;u;\nabla u) = a > 0$ -- the method only needs to use the strict upper bound of the diffusion coefficient, saying that $a < \overline{M}$ (reduced from $\left(\text{A}_{1}\right)$). Corresponding to $\left(\text{A}_{2}\right)$, this way we take $A = \overline{M} - a \in (\overline{M} - M_1,\overline{M})$ for $a< M_1 < \overline{M}$ by the completeness of real numbers. 
	\item $F(x,t;u;\nabla u) = F(x,t;u)$ is globally Lipschitz-continuous in $u$, i.e. $F_{\ell} = F$ and $L_F(\ell )= L_F$ is independent of all involved parameters; see $\left(\text{A}_{3}\right)$ and \eqref{eq:2.1} with a typical example $F(u)=\sin u$. We will come back to the locally Lipschitz-continuous case of $F$ in \cref{sec:remarks}. This case is significantly more difficult to estimate due to the blow-up profile of the cut-off parameter $\ell^{\varepsilon}$; see \eqref{cutoff-para}. 
\end{itemize}

Hereby, when we recall these assumptions, i.e. $\left(\text{A}_{1}\right)$-$\left(\text{A}_{4}\right)$, in the analysis, it is understood that the correspondingly reduced versions are considered. We below scrutinize the existence result for the regularized problem and the convergence analysis obtained after applying the QR scheme \eqref{reguproblem}-\eqref{reguproblem-1} to the semi-linear case \eqref{Pb1-new}. When doing so, proofs of our results are based on several energy-like estimates using an auxiliary parameter, denoted either by $\rho_{\varepsilon}$ or by $\rho_{\beta}$, depending on whether the regularization parameter $\beta$ is involved. In this spirit, we technically seek fine energy controls for the ``scaled" problems obtained by the weight function $e^{\rho_{\varepsilon}(t-T)}$. The choice of this parameter is definitely dependent on every single aspect of analysis, but it will at least include the magnitude of stability of the regularized problem. Thus, its behavior obeys
\[
\lim_{\varepsilon\to0^{+}}\rho_\varepsilon=\infty.
\]
\subsection{Existence result for the regularized problem}\label{subsec:4.2}
For each $\varepsilon>0$, we put $v^{\varepsilon}\left(x,t\right)=e^{\rho_{\varepsilon}\left(t-T\right)}u^{\varepsilon}\left(x,t\right)$. Under a suitable choice of such a parameter, we obtain the existence result for the regularized problem in the framework of Faedo-Galerkin procedures. Using  $\left(\text{A}_{3}\right)$, the regularized problem \eqref{reguproblem} for the semi-linear case \eqref{reguproblem-1} can be rewritten as
\begin{equation}\label{regune}
u_{t}^{\varepsilon}+A \Delta u^{\varepsilon}=F\left(x,t;u^{\varepsilon}\right)+\mathbf{P}_{\varepsilon}^{\beta}u^{\varepsilon}.
\end{equation}
Multiplying this equation by $e^{\rho_{\varepsilon}\left(t-T\right)}$,
it becomes
\begin{equation}
v_{t}^{\varepsilon}+A\Delta v^{\varepsilon}-\rho_{\varepsilon}v^{\varepsilon}=e^{\rho_{\varepsilon}\left(t-T\right)}F\left(x,t;u^{\varepsilon}\right)+\mathbf{P}_{\varepsilon}^{\beta}v^{\varepsilon},\label{eq:4.1}
\end{equation}
by virtue of the linearity of the operator $\mathbf{P}_{\varepsilon}^{\beta}$.

Note that the boundary and terminal conditions of \eqref{eq:4.1} remain the same as \eqref{reguproblem-1} due to the structural definition of $v^{\varepsilon}$. Henceforward, multiplying
(\ref{eq:4.1}) by a test function $\psi\in H_{0}^{1}\left(\Omega\right)$
we define a weak formulation of \eqref{reguproblem-1} in the following standard type.
\begin{definition}\label{def:weak} For each $\varepsilon > 0$, a function $v^{\varepsilon}$ is said to be a weak solution of (\ref{eq:4.1}) if
	\[v^{\varepsilon}\in L^{2}\left(0,T;H_{0}^{1}\left(\Omega\right)\right)\cap L^{\infty}\left(0,T;L^{2}\left(\Omega\right)\right)
	\]
and it holds
	\begin{align}
	\frac{d}{dt}\left\langle v^{\varepsilon},\psi\right\rangle  & -A\int_{\Omega}\nabla v^{\varepsilon}\cdot\nabla\psi dx-\rho_{\varepsilon}\left\langle v^{\varepsilon},\psi\right\rangle \label{eq:4.2} \\
	& =e^{\rho_{\varepsilon}\left(t-T\right)}\left\langle F\left(\cdot,t;e^{\rho_{\varepsilon}\left(T-t\right)}v^{\varepsilon}\right),\psi\right\rangle +\left\langle \mathbf{P}_{\varepsilon}^{\beta}v^{\varepsilon},\psi\right\rangle \;\text{for all } \psi\in H_{0}^{1}\left(\Omega\right).\nonumber
	\end{align}
\end{definition}

Let $\mathbb{S}_{n}$ be the space generated by $\phi_{1},\phi_{2},...,\phi_{n}$
for $n=1,2,...$ where in general $\left\{\phi_{j}\right\}$ is a Schauder basis of $H^1(\Omega)$ (so it can be the eigenfunctions mentioned in \cref{exampleope}), then let
\begin{equation}
v_{n}^{\varepsilon}\left(x,t\right)=\sum_{j=1}^{n}V_{jn}^{\varepsilon}\left(t\right)\phi_{j}\left(x\right)\label{eq:4.5}
\end{equation}
be the weak solution of the following approximate problem, corresponding to \eqref{eq:4.1}:
\begin{align}\label{eq:4.3}
	\left\langle \left(v_{n}^{\varepsilon}\right)_{t},\psi\right\rangle  & -A\int_{\Omega}\nabla v_{n}^{\varepsilon}\cdot\nabla\psi dx-\rho_{\varepsilon}\left\langle v_{n}^{\varepsilon},\psi\right\rangle  \\
	& =e^{\rho_{\varepsilon}\left(t-T\right)}\left\langle F\left(\cdot,t;e^{\rho_{\varepsilon}\left(T-t\right)}v_{n}^{\varepsilon}\right),\psi\right\rangle +\left\langle \mathbf{P}_{\varepsilon}^{\beta}v_{n}^{\varepsilon},\psi\right\rangle ,\nonumber
\end{align}
for all $\psi\in\mathbb{S}_{n}$, with the final condition
\begin{equation}
v_{n}^{\varepsilon}\left(T\right)=v_{fn}^{\varepsilon}=\sum_{j=1}^{n}\left(V_{f}^{\varepsilon}\right)_{jn}\phi_{j}\to u_{f}^{\varepsilon}\;\text{ strongly in }L^{2}\left(\Omega\right)\;\text{ as }n\to \infty.\label{eq:4.4}
\end{equation}

To derive the nonlinear ordinary differential equations with respect to the time argument for $V_{jn}\left(t\right)$, it follows from (\ref{eq:4.3}) with using $\psi = \phi_{j}$ that for $1\le j\le n$,
\begin{align*}
	\left(V_{jn}^{\varepsilon}\right)_{t} -(A+\rho_{\varepsilon})V_{jn}^{\varepsilon}
	=e^{\rho_{\varepsilon}\left(t-T\right)}\left\langle F\left(\cdot,t;e^{\rho_{\varepsilon}\left(T-t\right)}v_{n}^{\varepsilon}\right),\phi_{j}\right\rangle 
	+\left\langle \mathbf{P}_{\varepsilon}^{\beta}v_{n}^{\varepsilon},\phi_{j}\right\rangle ,
\end{align*}
and $V_{jn}^{\varepsilon}\left(T\right)=\left(V_{f}^{\varepsilon}\right)_{jn}$.

By using the Newton-Liebniz formula, one has
\begin{align}
	 V_{jn}^{\varepsilon}\left(t\right) &  =\left(V_{f}^{\varepsilon}\right)_{jn}-(A+\rho_{\varepsilon})\int_{t}^{T}V_{jn}^{\varepsilon}\left(s\right)ds \label{eq:4.6}\\
	& -\int_{t}^{T}\left[e^{\rho_{\varepsilon}\left(s-T\right)}\left\langle F\left(\cdot,s;e^{\rho_{\varepsilon}\left(T-s\right)}v_{n}^{\varepsilon}\right),\phi_{j}\right\rangle +\left\langle \mathbf{P}_{\varepsilon}^{\beta}v_{n}^{\varepsilon}\left(s\right),\phi_{j}\right\rangle \right]ds.\nonumber
\end{align}

\begin{lemma}
	Suppose that (\ref{eq:4.1-1}) holds. For any fixed
	$n\in\mathbb{N}$ and for each $\varepsilon>0$, the system (\ref{eq:4.3})-(\ref{eq:4.4})
	has a unique solution $V_{jn}^{\varepsilon}\in C([0,T])$.
\end{lemma}
\begin{proof}
	The proof of this lemma is standard. Here we sketch out some important steps because it seems pertinent to see more detailed impact of $\rho_{\varepsilon}$ on all the analysis. We define the norm in the Banach space $Y=C\left(\left[0,T\right];\mathbb{R}^{n}\right)$ as follows:
	\[
	\left\Vert c\right\Vert _{Y}:=\sup_{t\in\left[0,T\right]}\sum_{j=1}^{n}\left|c_{j}\left(t\right)\right|\quad\text{with }c=\left(c_{j}\right)_{1\le j\le n}.
	\]
	By virtue of (\ref{eq:4.6}), we can define a Volterra-type integral equation and then set the operator $\mathcal{G}:C\left(\left[0,T\right];\mathbb{R}^{n}\right)\to C\left(\left[0,T\right];\mathbb{R}^{n}\right)$
	by
	\[
	\mathcal{G}\left(V^{\varepsilon}\right)\left(t\right)=H^{\varepsilon}-\int_{t}^{T}K^{\varepsilon}\left(s,V^{\varepsilon}\right)ds,
	\]
	where in the vector form, $V^{\varepsilon}$ and $H^{\varepsilon}$ indicate $V_{j}^{\varepsilon}$ and $\left(V_{f}^{\varepsilon}\right)_{j}$, respectively, and $K^{\varepsilon}$ stands for the right-hand side of (\ref{eq:4.6}) under the integration in time.
	
	Observe that when summing (\ref{eq:4.6}) with respect to $j$ up to $n$, we have
	\begin{align}
		\sum_{j=1}^{n}V_{j}^{\varepsilon}\left(t\right)& =\sum_{j=1}^{n}\left(V_{f}^{\varepsilon}\right)_{j}-(A+\rho_{\varepsilon})\sum_{j=1}^{n}\int_{t}^{T}V_{j}^{\varepsilon}\left(s\right)ds\\
		& -\sum_{j=1}^{n}\int_{t}^{T}\left[e^{\rho_{\varepsilon}\left(s-T\right)}\left\langle F\left(\cdot,s;e^{\rho_{\varepsilon}\left(T-s\right)}v_{n}^{\varepsilon}\right),\phi_{j}\right\rangle +\left\langle \mathbf{P}_{\varepsilon}^{\beta}v_{n}^{\varepsilon}\left(s\right),\phi_{j}\right\rangle \right]ds.\nonumber \label{eq:4.7}
	\end{align}

	For $V^{\varepsilon}\in C\left(\left[0,T\right];\mathbb{R}^{n}\right)$
and $W^{\varepsilon}\in C\left(\left[0,T\right];\mathbb{R}^{n}\right)$ we have the following estimates.
With the aid of the Lipschitz assumption $\left(\text{A}_{3}\right)$  we easily get
\begin{align}\label{moine1}
	\left|\left\langle F\left(\cdot,s;e^{\rho_{\varepsilon}\left(T-s\right)}v_{n}^{\varepsilon}\right)-F\left(\cdot,s;e^{\rho_{\varepsilon}\left(T-s\right)}w_{n}^{\varepsilon}\right),\phi_{j}\right\rangle\right|  \le CL_{F}e^{\rho_{\varepsilon}\left(T-s\right)}\sum_{k=1}^{n}\left|V_{k}^{\varepsilon}-W_{k}^{\varepsilon}\right|,
\end{align}
and in the same vein, using (\ref{eq:4.1-1}) implies that
\begin{equation}\label{moine2}
\left|\left\langle \mathbf{P}_{\varepsilon}^{\beta}v_{n}^{\varepsilon}\left(s\right),\phi_{j}\right\rangle -\left\langle \mathbf{P}_{\varepsilon}^{\beta}w_{n}^{\varepsilon}\left(s\right),\phi_{j}\right\rangle\right| \le CC_{1}\log\left(\gamma\left(T,\beta\right)\right)\sum_{k=1}^{n}\left|V_{k}^{\varepsilon}-W_{k}^{\varepsilon}\right|.
\end{equation}

	Grouping \eqref{moine1} and \eqref{moine2}, it follows from \eqref{eq:4.7} that the following estimate can be obtained:
	\begin{align*}
	& \left|\mathcal{G}\left(V^{\varepsilon}\right)-\mathcal{G}\left(W^{\varepsilon}\right)\right| \\ &\le C\left(T-t\right)\left(\overline{M}+\rho_{\varepsilon}+nL_{F}+C_{1}\log\left(\gamma\left(T,\beta\right)\right)\right)\left\Vert V^{\varepsilon}-W^{\varepsilon}\right\Vert _{Y},
	\end{align*}
	and furthermore, by induction we deduce
	\begin{align*}
	&
	\left|\mathcal{G}^{m}\left(V^{\varepsilon}\right)-\mathcal{G}^{m}\left(W^{\varepsilon}\right)\right|
	\\& \le\frac{\left(T-t\right)^{m}}{m!}C^{m}\left(\overline{M}+\rho_{\varepsilon}+L_{F}+C_{1}\log\left(\gamma\left(T,\beta\right)\right)\right)^{m}\left\Vert V^{\varepsilon}-W^{\varepsilon}\right\Vert _{Y},
	\end{align*}
	where we denote by $\mathcal{G}^{m}\left(V^{\varepsilon}\right)=\mathcal{G}\left(\mathcal{G}...\mathcal{G}\left(V^{\varepsilon}\right)\right)$.
	
	Since for each $\varepsilon>0$ and $n\in\mathbb{N}$, there exists $m_{0}\in\mathbb{N}$ such that 
	\[
	\frac{\left(T-t\right)^{m_{0}}}{m_{0}!}C^{m_{0}}\left(\overline{M}+\rho_{\varepsilon}+L_{F}+C_{1}\log\left(\gamma\left(T,\beta\right)\right)\right)^{m_{0}}<1,
	\]
	then $\mathcal{G}^{m_{0}}$ is a contraction mapping from $C\left(\left[0,T\right];\mathbb{R}^{n}\right)$
	onto itself. By the Banach fixed-point argument, there exists a unique
	solution $V^{\varepsilon}$ in $Y$ such that $\mathcal{G}^{m_{0}}\left(V^{\varepsilon}\right)=V^{\varepsilon}$.
	Combining this with the fact that $\mathcal{G}^{m_{0}}\left(\mathcal{G}\left(V^{\varepsilon}\right)\right)=\mathcal{G}\left(\mathcal{G}^{m_{0}}\left(V^{\varepsilon}\right)\right)=\mathcal{G}\left(V^{\varepsilon}\right)$,
	the integral equation $\mathcal{G}\left(V^{\varepsilon}\right)=V^{\varepsilon}$
	admits a unique solution in $C\left(\left[0,T\right];\mathbb{R}^{n}\right)$.
\end{proof}

From here on, we state the existence result in the following theorem.
\begin{theorem}\label{existencetheorem}
	For each $\varepsilon>0$, the regularized problem (\ref{eq:4.1})
	has a weak solution $v^{\varepsilon}$ in the sense of \cref{def:weak}. Moreover, it satisfies $v^{\varepsilon}\in C\left(\left[0,T\right];L^{2}\left(\Omega\right)\right)$ and  $v_{t}^{\varepsilon}\in L^{2}\left(0,T;\left(H^{1}\left(\Omega\right)\right)'\right)$. 
\end{theorem}
\begin{proof}
	We now give some \emph{a priori} estimates for the solution of the problem (\ref{eq:4.1}). When doing so, we choose $\psi=v_{n}^{\varepsilon}$ in (\ref{eq:4.3}) to get
	\begin{align}\label{eq:4.11}
		\frac{1}{2}\frac{d}{dt}\left\Vert v_{n}^{\varepsilon}\right\Vert _{L^{2}\left(\Omega\right)}^{2} &-A\left\Vert \nabla v_{n}^{\varepsilon}\right\Vert _{[L^{2}\left(\Omega\right)]^{d}}^{2}-\rho_{\varepsilon}\left\Vert v_{n}^{\varepsilon}\right\Vert _{L^{2}\left(\Omega\right)}^{2} \\
		 = & \underbrace{e^{\rho_{\varepsilon}\left(t-T\right)}\left\langle F\left(\cdot,t;e^{\rho_{\varepsilon}\left(T-t\right)}v_{n}^{\varepsilon}\right),v_{n}^{\varepsilon}\right\rangle }_{:=I_{3}}+\underbrace{\left\langle \mathbf{P}_{\varepsilon}^{\beta}v_{n}^{\varepsilon},v_{n}^{\varepsilon}\right\rangle }_{:=I_{4}}.\nonumber
	\end{align}

	Note from the resulting structural condition of $F$
	in $\left(\text{A}_{3}\right)$ that
	\[
	e^{\rho_{\varepsilon}\left(t-T\right)}\left|F\left(x,t;e^{\rho_{\varepsilon}\left(T-t\right)}v_{n}^{\varepsilon}\right)-F\left(x,t;0\right)\right|\le L_{F}\left|v_{n}^{\varepsilon}\right|,
	\]
	one thus has
	\begin{align*}
		I_{3} & \ge-\frac{e^{2\rho_{\varepsilon}\left(t-T\right)}}{2L_{F}}\left\Vert F\left(\cdot,t;e^{\rho_{\varepsilon}\left(T-t\right)}v_{n}^{\varepsilon}\right)\right\Vert _{L^{2}\left(\Omega\right)}^{2}-\frac{L_{F}}{2}\left\Vert v_{n}^{\varepsilon}\right\Vert _{L^{2}\left(\Omega\right)}^{2}\\
		& \ge-\frac{e^{2\rho_{\varepsilon}\left(t-T\right)}}{2L_{F}}\left\Vert F\left(\cdot,t;0\right)\right\Vert _{L^{2}\left(\Omega\right)}^{2}
		-\frac{1}{2}\left(1+L_{F}\right)\left\Vert v_{n}^{\varepsilon}\right\Vert ^{2}_{L^{2}\left(\Omega\right)}.
	\end{align*}

	Similarly, based on the structural definition of $\mathbf{P}_{\varepsilon}^{\beta}$
	in (\ref{eq:4.1-1}), it yields
	\begin{align*}
		\left\langle \mathbf{P}_{\varepsilon}^{\beta}v_{n}^{\varepsilon},v_{n}^{\varepsilon}\right\rangle  & \ge-\frac{1}{2}\left(\left\Vert \mathbf{P}_{\varepsilon}^{\beta}v_{n}^{\varepsilon}\right\Vert _{L^{2}\left(\Omega\right)}^{2}+\left\Vert v_{n}^{\varepsilon}\right\Vert _{L^{2}\left(\Omega\right)}^{2}\right)\\
		& \ge-\frac{1}{2}\left(C_{1}^{2}\log^{2}\left(\gamma\left(T,\beta\right)\right)+1\right)\left\Vert v_{n}^{\varepsilon}\right\Vert _{L^{2}\left(\Omega\right)}^{2}.
	\end{align*}

	Then, (\ref{eq:4.11}) can be estimated by
	\begin{align*}
		& \frac{d}{dt}\left\Vert v_{n}^{\varepsilon}\right\Vert _{L^{2}\left(\Omega\right)}^{2}+\frac{e^{2\rho_{\varepsilon}\left(t-T\right)}}{L_{F}}\left\Vert F\left(\cdot,t;0\right)\right\Vert _{L^{2}\left(\Omega\right)}^{2}\\
		& \ge 2\underline{M}\left\Vert \nabla v_{n}^{\varepsilon}\right\Vert _{\left[L^{2}\left(\Omega\right)\right]^{d}}^{2}
		+\left(2\rho_{\varepsilon}-(1+L_F)-C_{1}^{2}\log^{2}\left(\gamma\left(T,\beta\right)\right)-1\right)\left\Vert v_{n}^{\varepsilon}\right\Vert _{L^{2}\left(\Omega\right)}^{2},
	\end{align*}
	where we have used the assumption $\left(\text{A}_{3}\right)$.
	
	Hereby, for each $\varepsilon>0$ we choose $2\rho_{\varepsilon}=L_{F}+C_{1}^{2}\log^{2}\left(\gamma\left(T,\beta\right)\right)+2>0$,
	then integrate the resulting estimate from $t$ to $T$ to obtain
	\begin{align*}
		\left\Vert v_{n}^{\varepsilon}\left(T\right)\right\Vert _{L^{2}\left(\Omega\right)}^{2} & +\frac{e^{-2T\rho_{\varepsilon}}}{L_{F}}\int_{t}^{T}\left\Vert F\left(\cdot,s;0\right)\right\Vert _{L^{2}\left(\Omega\right)}^{2}ds\\
		& \ge\left\Vert v_{n}^{\varepsilon}\left(t\right)\right\Vert _{L^{2}\left(\Omega\right)}^{2}+\underline{M}\int_{t}^{T}\left\Vert \nabla v_{n}^{\varepsilon}(s)\right\Vert _{[L^{2}\left(\Omega\right)]^{d}}^{2}ds.
	\end{align*}

	Since $v_{n}^{\varepsilon}\left(T\right)\to u_{f}^{\varepsilon}$
	in $L^{2}\left(\Omega\right)$ (cf. (\ref{eq:4.4})), we can find an $\varepsilon$-independent constant $\bar{c}>0$ such
	that
	\[
	\left\Vert v_{n}^{\varepsilon}\left(t\right)\right\Vert _{L^{2}\left(\Omega\right)}^{2}+\underline{M}\int_{t}^{T}\left\Vert \nabla v_{n}^{\varepsilon}\right\Vert _{[L^{2}\left(\Omega\right)]^{d}}^{2}ds\le\bar{c}.
	\]
	
	As byproduct, we have
	\begin{equation}
	v_{n}^{\varepsilon}\;\text{is bounded in }L^{\infty}\left(0,T;L^{2}\left(\Omega\right)\right)\;\text{and in }L^{2}\left(0,T;H_{0}^{1}\left(\Omega\right)\right).\label{eq:4.11-1}
	\end{equation}
	
	Observe that
	\begin{align*}
		\left(v_{n}^{\varepsilon}\right)_{t} & +A\Delta v_{n}^{\varepsilon}-\rho_{\varepsilon}v_{n}^{\varepsilon} =e^{\rho_{\varepsilon}\left(t-T\right)}F\left(x,t;e^{\rho_{\varepsilon}\left(T-t\right)}v_{n}^{\varepsilon}\right)+\mathbf{P}_{\varepsilon}^{\beta}v_{n}^{\varepsilon}\in\left(H^{1}\left(\Omega\right)\right)',
	\end{align*}
	which provides
	\begin{equation}
	\left(v_{n}^{\varepsilon}\right)_{t}\;\text{is bounded in }L^{2}\left(0,T;\left(H^{1}\left(\Omega\right)\right)'\right).\label{eq:4.12}
	\end{equation}
	
	Thanks to the Banach-Alaoglu theorem, the uniform bounds with respect
	to $n$, as obtained in (\ref{eq:4.11-1})-(\ref{eq:4.12}), imply
	that one can extract a subsequence (which we relabel with the index
	$n$ if necessary) such that for each $\varepsilon>0$,
	\begin{equation}
	v_{n}^{\varepsilon}\to v^{\varepsilon}\;\text{weakly}-*\;\text{in }L^{\infty}\left(0,T;L^{2}\left(\Omega\right)\right),\label{eq:4.13-1}
	\end{equation}
	\begin{equation}
	v_{n}^{\varepsilon}\to v^{\varepsilon}\;\text{weakly in }L^{2}\left(0,T;H_{0}^{1}\left(\Omega\right)\right),\label{eq:4.13}
	\end{equation}
	\begin{equation}
	\left(v_{n}^{\varepsilon}\right)_{t}\to v_{t}^{\varepsilon}\;\text{weakly in }L^{2}\left(0,T;\left(H^{1}\left(\Omega\right)\right)'\right).\label{eq:4.14}
	\end{equation}
	Furthermore, by the Aubin-Lions compactness theorem in combination
	with the Gelfand triple $H_{0}^{1}\left(\Omega\right)\subset L^{2}\left(\Omega\right)\subset\left(H^{1}\left(\Omega\right)\right)'$,
	one gets from (\ref{eq:4.13-1}) and (\ref{eq:4.14}) that
	\begin{equation}
	v_{n}^{\varepsilon}\to v^{\varepsilon}\;\text{strongly in }L^{2}\left(Q_T\right)\text{and so a.e. in \ensuremath{Q_T} for a further subsequence}.\label{eq:4.16-3}
	\end{equation}
	
	Note also that due to (\ref{eq:4.1-1}), one has for each $\varepsilon>0$,
	\begin{equation}
	\mathbf{P}_{\varepsilon}^{\beta}v_{n}^{\varepsilon}\to\mathbf{P}_{\varepsilon}^{\beta}v^{\varepsilon}\;\text{strongly in }L^{2}\left(Q_{T}\right)\text{and so a.e. in \ensuremath{Q_{T}} for a further subsequence}.\label{eq:4.16-2}
	\end{equation}
	In the same manner, one has for each $\varepsilon>0$,
	\begin{equation}
	F\left(e^{\rho_{\varepsilon}\left(T-t\right)}v_{n}^{\varepsilon}\right)\to F\left(e^{\rho_{\varepsilon}\left(T-t\right)}v^{\varepsilon}\right)\;\text{strongly in }L^{2}\left(Q_T\right).
	\label{eq:4.16-1}
	\end{equation}
	
	
	From here on, by grouping (\ref{eq:4.13-1})-(\ref{eq:4.14}) and
	(\ref{eq:4.16-3})-(\ref{eq:4.16-1}) we can pass to the limit in (\ref{eq:4.3})
	to show that $v^{\varepsilon}$ satisfies the problem (\ref{eq:4.1}) in the weak sense \eqref{eq:4.2}.  On top of that, due to (\ref{eq:4.11-1}) and (\ref{eq:4.12}), we
	have
	\[
	v^{\varepsilon}\in C\left(\left[0,T\right];L^{2}\left(\Omega\right)\right),
	\]
	where we have applied the embeddings $H_{0}^{1}\left(\Omega\right)\subset L^{2}\left(\Omega\right)\subset\left(H^{1}\left(\Omega\right)\right)'$
	and $H^{1}\left(0,T\right)\subset C\left[0,T\right]$.
	
	Now, it remains to verify the terminal data. In fact, we take a function
	$\vartheta\in C^{1}\left[0,T\right]$ with $\vartheta\left(0\right)=0$
	and $\vartheta\left(T\right)=1$. As a consequence of the convergence
	(\ref{eq:4.12}), one has
	\[
	\int_{0}^{T}\left\langle \left(v_{n}^{\varepsilon}\right)_{t},\psi\right\rangle \vartheta dt\to\int_{0}^{T}\left\langle v_{t}^{\varepsilon},\psi\right\rangle \vartheta dt\quad\text{for all }\psi\in L^{2}\left(\Omega\right),
	\]
	and by integration by parts together with the Newton-Liebniz formula,
	it becomes
	\begin{equation}
	-\int_{0}^{T}\left\langle v_{n}^{\varepsilon},\psi\right\rangle \vartheta_{t}dt+\left\langle v_{n}^{\varepsilon}\left(T\right),\psi\right\rangle \vartheta\left(T\right)\to-\int_{0}^{T}\left\langle v^{\varepsilon},\psi\right\rangle \vartheta_{t}dt+\left\langle v^{\varepsilon}\left(T\right),\psi\right\rangle \vartheta\left(T\right),\label{eq:4.20}
	\end{equation}
	for all $\psi\in L^{2}\left(\Omega\right)$.
	Consequently, the weak convergence (\ref{eq:4.13}) allows us to obtain
	$\left\langle v_{n}^{\varepsilon}\left(T\right),\psi\right\rangle \to\left\langle v^{\varepsilon}\left(T\right),\psi\right\rangle $
	for all $\psi\in H_{0}^{1}\left(\Omega\right)$ from
	(\ref{eq:4.20}). Combining this convergence with the fact already
	known that $v_{n}^{\varepsilon}\left(T\right)$ converges strongly
	to $u_{f}^{\varepsilon}$ in $L^{2}\left(\Omega\right)$;
	see (\ref{eq:4.4}). We thus get $\left\langle v_{n}^{\varepsilon}\left(T\right),\psi\right\rangle \to\left\langle u_{f}^{\varepsilon},\psi\right\rangle $
	for all $\psi\in H_{0}^{1}\left(\Omega\right)$. Due
	to the uniqueness of the limit, it reveals that $\left\langle v^{\varepsilon}\left(T\right),\psi\right\rangle =\left\langle u_{f}^{\varepsilon},\psi\right\rangle $
	for all $\psi\in H_{0}^{1}\left(\Omega\right)$ and
	thus $v^{\varepsilon}\left(T\right)=u_{f}^{\varepsilon}$ a.e. in
	$\Omega$.
\end{proof}

Now we show the positivity and boundedness of solution to the regularized problem \eqref{eq:4.1}. In the following theorem, if the measured inputs of the concentrations are positive and essentially bounded in a spatial environment, their distributions that obey the proposed approximation remain the same properties therein by a suitable choice of the auxiliary parameter $\rho_{\varepsilon}$. In other words, the behavior of the regularized solution strictly depends on the way $\rho_{\varepsilon}$ being taken.

\begin{theorem}\label{posi}
	Let $v^{\varepsilon}$ be a weak solution of the problem \eqref{eq:4.1} as deduced
	in \cref{existencetheorem}. For each $\varepsilon>0$, suppose that $0\le u_{f}^{\varepsilon}\in L^{\infty}\left(\Omega\right)$
	and $F\left(x,t;0\right)\equiv0$ for a.e. $\left(x,t\right)\in Q_{T}$. Moreover, for all real-valued constant $C>0$ we assume $\mathbf{P}_{\varepsilon}^{\beta}C=\mathbf{Q}_{\varepsilon}^{\beta}C \ge 0$. Then, $0\le v^{\varepsilon}\le\left\Vert u_{f}^{\varepsilon}\right\Vert _{L^{\infty}\left(\Omega\right)}$
	for a.e. $\left(x,t\right)\in Q_{T}$.
\end{theorem}
\begin{proof}
	Let $v^{\varepsilon}:=v^{\varepsilon,+}-v^{\varepsilon,-}$ where
	$f^{+}:=\max\left\{ f,0\right\} $ and $f^{-}:=\max\left\{ -f,0\right\} $.
	In \eqref{eq:4.2}, we now take the test function $\psi=-v^{\varepsilon,-}$.
	Then, by $\left(\text{A}_{3}\right)$, $\left(\text{A}_{4}\right)$ and \eqref{eq:4.1-1} we have
	\begin{align}
		\frac{d}{dt}\left\Vert v^{\varepsilon,-}\right\Vert _{L^{2}\left(\Omega\right)}^{2} & \ge\underline{M}\left\Vert \nabla v^{\varepsilon,-}\right\Vert _{\left[L^{2}\left(\Omega\right)\right]^{d}}^{2}\label{eq:positive}\\
		& +\left(\rho_{\varepsilon}-L_{F}-C_{1}\log\left(\gamma\left(T,\beta\right)\right)-1\right)\left\Vert v^{\varepsilon,-}\right\Vert _{L^{2}\left(\Omega\right)}^{2},
		\nonumber 
	\end{align}
	inspired very much the way we have estimated \eqref{eq:4.11}.
	
	Choosing $\rho_{\varepsilon}=L_{F}+C_{1}\log\left(\gamma\left(T,\beta\right)\right)+1>0$
	and observing that $\left.v^{\varepsilon,-}\right|_{t=T}=0$, we integrate
	\eqref{eq:positive} from $t$ to $T$ to get $\left\Vert v^{\varepsilon,-}\right\Vert _{L^{2}\left(\Omega\right)}^{2}\le0$,
	which indicates the positivity of $v^{\varepsilon}$.
	
	To prove the upper bound, we take the test function $\psi=\left(v^{\varepsilon}-B\right)^{+}$
	in \eqref{eq:4.2} where $B\ge\left\Vert u_{f}^{\varepsilon}\right\Vert _{L^{\infty}\left(\Omega\right)}$.
	Thus, we arrive at
	\begin{align}\label{eq:bounded}
		& \frac{d}{dt}\left\Vert \left(v^{\varepsilon}-B\right)^{+}\right\Vert _{L^{2}\left(\Omega\right)}^{2}\\
		& \ge\underline{M}\left\Vert \nabla\left(v^{\varepsilon}-B\right)^{+}\right\Vert _{\left[L^{2}\left(\Omega\right)\right]^{d}}^{2}+\rho_{\varepsilon}\left\Vert \left(v^{\varepsilon}-B\right)^{+}\right\Vert _{L^{2}\left(\Omega\right)}^{2}+\rho_{\varepsilon}\left\langle B,\left(v^{\varepsilon}-B\right)^{+}\right\rangle \nonumber\\
		& +\left\langle \mathbf{P}_{\varepsilon}^{\beta}\left(v^{\varepsilon}-B\right)^{+},\left(v^{\varepsilon}-B\right)^{+}\right\rangle +\left\langle \mathbf{P}_{\varepsilon}^{\beta}B,\left(v^{\varepsilon}-B\right)^{+}\right\rangle \nonumber\\
		& +\underbrace{e^{\rho_{\varepsilon}\left(t-T\right)}\left\langle F\left(e^{\rho_{\varepsilon}\left(T-t\right)}v^{\varepsilon}\right),\left(v^{\varepsilon}-B\right)^{+}\right\rangle}_{:= I_{5}}.\nonumber
	\end{align}

	Here, taking into account the structural condition of $F$ we get
	\begin{align*}
		I_{5} & \ge-L_{F}\left|\left\langle \left|v^{\varepsilon}\right|,\left(v^{\varepsilon}-B\right)^{+}\right\rangle \right|\\
		& \ge-L_{F}\left(\left\Vert \left(v^{\varepsilon}-B\right)^{+}\right\Vert _{L^{2}\left(\Omega\right)}^{2}+\left\langle B,\left(v^{\varepsilon}-B\right)^{+}\right\rangle \right).
	\end{align*}

	At this stage, we proceed as in the proof of the positivity. By choosing $
	\rho_{\varepsilon}=C_{1}\log\left(\gamma\left(T,\beta\right)\right)+L_{F}>0$, it follows from \eqref{eq:bounded} that $\left(v^{\varepsilon}-B\right)^{+}=0$, provided that $\left.\left(v^{\varepsilon}-B\right)^{+}\right|_{t=T}=0$. Hence, we complete the proof of the theorem.
\end{proof}

\subsection{Convergence analysis}
\label{sec:5}
We are now going to derive the convergence rates obtained when the regularized solution $u_{\beta}^{\varepsilon}$ of \eqref{reguproblem}-\eqref{reguproblem-1} is applied to approximate the solution $u$ of \eqref{Pb1-new}-\eqref{terminal} in the presence of noise on the final data. Note that in the previous subsection we only write $u^{\varepsilon}$ as the regularized solution since the parameter $\varepsilon$ is already fixed. Instead, we denote in this part $u_{\beta}^{\varepsilon}$ due to the choice of the regularization parameter $\beta(\varepsilon)$ which plays a vital role in this analysis.

Although \cref{exampleope} shows that $C_1 = \frac{1}{T}$, for an arbitrary $C_{1}>0$ we need $C_{1}T \le 1$ in our main results to gain strong convergences. At some points, this is in the same spirit of the terminology \emph{small solution} defined in \cite{Cara77}.

\subsubsection{Statement of the results}
Here we state our main results as \cref{thm:Estimate1} and \cref{thm:Estimate2}; their solid proofs are
deferred to \cref{subsec:5.2} and \cref{subsec:5.3}, respectively. Moreover, proof of \cref{cor:5.3} is given in \cref{subsec:5.4}.

In the following, let $\gamma\left(t,\beta\right)$ for $t\in\left[0,T\right]$
and $\beta:=\beta\left(\varepsilon\right)$ be
as in \cref{sec:4}. We choose
\begin{equation}
{\displaystyle \lim_{\varepsilon\to0^{+}}\gamma^{C_{1}T}\left(T,\beta\right)\varepsilon=K\in\left(0,\infty\right)}.\label{blowup}
\end{equation}

\begin{theorem}
	\textbf{\emph{\label{thm:Estimate1}(Error estimates for $0<t\le T$)}}
	
	Assume that the problem \eqref{Pb1-new}-\eqref{terminal} admits a unique solution
	\begin{equation}
	u\in C\left(\left[0,T\right];\mathbb{W}\right),\label{eq:dku}
	\end{equation}
	where the precise structure of $\mathbb{W}$ depends on the choice
	of the operator $\mathbf{Q}_{\varepsilon}^{\beta}$ in (\ref{eq:4.1-3}). For a suitable choice of the operator $\mathbf{P}_{\varepsilon}^{\beta}$
	in (\ref{eq:4.1-1}), we consider $u_{\beta}^{\varepsilon}\in C\left(\left[0,T\right];L^{2}\left(\Omega\right)\right)$
	as a solution of \eqref{regune}-\eqref{reguproblem-1} corresponding
	to the measured data $u_{f}^{\varepsilon}$. Then the following error estimate holds
	\begin{align*}
		& \left\Vert u_{\beta}^{\varepsilon}\left(\cdot,t\right)-u\left(\cdot,t\right)\right\Vert _{L^{2}\left(\Omega\right)} +\sqrt{2\underline{M}} \int_{t}^{T}\left\Vert \nabla u_{\beta}^{\varepsilon}\left(\cdot,s\right)-\nabla u\left(\cdot,s\right)\right\Vert _{\left[L^{2}\left(\Omega\right)\right]^{d}}ds\\
		\le & \gamma^{-C_{1}t}\left(T,\beta\right)\left(K+\sqrt{2T}C_{0}\gamma^{C_{1}T-1}\left(T,\beta\right)\left\Vert u\right\Vert _{C\left(\left[0,T\right];\mathbb{W}\right)}\right)e^{TC_{2}},
	\end{align*}
	for $t\in\left(0,T\right)$ and $C_{i}>0$ ($i\in\left\{ 0,1,2\right\} $)
	independent of $\varepsilon$.
\end{theorem}

\begin{theorem}
	\textbf{\emph{\label{thm:Estimate2}(Error estimate for $t=0$)}}
	
	Under the assumptions of \cref{thm:Estimate1}, we assume further that
	\begin{equation}
	u\in C\left(\left[0,T\right];\mathbb{W}\right)\cap C^{1}\left(0,T;L^{2}\left(\Omega\right)\right).\label{eq:4.37-1}
	\end{equation}
	Then, for $\varepsilon>0$
	small enough we can find a unique $t^{\varepsilon}\in\left(0,T\right)$ such that
	\begin{align*}
		\left\Vert u_{\beta}^{\varepsilon}\left(\cdot,t^{\varepsilon}\right)-u\left(\cdot,0\right)\right\Vert _{L^{2}\left(\Omega\right)} &\le \left[\left(K+\sqrt{2T}C_{0}\gamma^{C_{1}T-1}\left(T,\beta\right)\left\Vert u\right\Vert _{C\left(\left[0,T\right];\mathbb{W}\right)}\right)e^{TC_{2}}\right.\\
		& \left.+\left\Vert u_{t}\right\Vert _{C\left(0,T;L^{2}\left(\Omega\right)\right)}\right]\frac{1}{\sqrt{C_{1}}\log^{\frac{1}{2}}\left(\gamma\left(T,\beta\right)\right)},
	\end{align*}
	where $C_{i}>0$ ($i\in\left\{ 0,1,2\right\}$) are independent of
	$\varepsilon$.
\end{theorem}

\begin{corollary}\label{cor:5.3}
	Under the assumptions of \cref{thm:Estimate1}, one has for any $0<t<T$, $u^{\varepsilon}_{\beta}$ is strongly convergent to $u$ in $L^{2}(t,T;L^r(\Omega))$ for some $r>2$ with the same rate as in \cref{thm:Estimate1}.
\end{corollary}

\subsubsection{Proof of \cref{thm:Estimate1}}\label{subsec:5.2}

For an auxiliary parameter $\rho_{\beta}>0$, we put $w_{\beta}^{\varepsilon}\left(x,t\right):=e^{\rho_{\beta}\left(t-T\right)}\left[u_{\beta}^{\varepsilon}\left(x,t\right)-u\left(x,t\right)\right]$.
Then, we compute that
\begin{align}
	\frac{\partial w_{\beta}^{\varepsilon}}{\partial t} & +A\Delta w_{\beta}^{\varepsilon}-\rho_{\beta}w_{\beta}^{\varepsilon}\label{eq:4.24} \\
	& =\mathbf{P}_{\varepsilon}^{\beta}w_{\beta}^{\varepsilon}+e^{\rho_{\beta}\left(t-T\right)}\mathbf{Q}_{\varepsilon}^{\beta}u +e^{\rho_{\beta}\left(t-T\right)}\left[F\left(x,t;u_{\beta}^{\varepsilon}\right)-F\left(x,t;u\right)\right].\nonumber
\end{align}
This equation is associated with the zero Dirichlet boundary condition
$w_{\beta}^{\varepsilon}=0$ on $\partial\Omega\times\left(0,T\right)$
and the following terminal condition:
\[
w_{\beta}^{\varepsilon}\left(x,T\right)=u_{\beta f}^{\varepsilon}\left(x\right)-u_{f}\left(x\right)\quad\text{for }x\in\Omega.
\]

Multiplying (\ref{eq:4.24}) by $w_{\beta}^{\varepsilon}$ and then
integrating the resulting equation over $\Omega$, we arrive at
\begin{align}\label{eq:4.25}
	& \frac{1}{2}\frac{d}{dt}\left\Vert w_{\beta}^{\varepsilon}\right\Vert _{L^{2}\left(\Omega\right)}^{2}-A\left\Vert \nabla w_{\beta}^{\varepsilon}\right\Vert _{\left[L^{2}\left(\Omega\right)\right]^{d}}^{2}-\rho_{\beta}\left\Vert w_{\beta}^{\varepsilon}\right\Vert _{L^{2}\left(\Omega\right)}^{2}\\
	& =\underbrace{\left\langle \mathbf{P}_{\varepsilon}^{\beta}w_{\beta}^{\varepsilon},w_{\beta}^{\varepsilon}\right\rangle }_{:=\mathcal{I}_{1}}+\underbrace{e^{\rho_{\beta}\left(t-T\right)}\left\langle \mathbf{Q}_{\varepsilon}^{\beta}u,w_{\beta}^{\varepsilon}\right\rangle }_{:=\mathcal{I}_{2}}+\underbrace{e^{\rho_{\beta}\left(t-T\right)}\left\langle F\left(u_{\beta}^{\varepsilon}\right)-F\left(u\right),w_{\beta}^{\varepsilon}\right\rangle }_{:=\mathcal{I}_{3}}.\nonumber
\end{align}

To investigate the convergence analysis, we need to bound from below the right-hand side of (\ref{eq:4.25}). Relying
on the structural property of the operator $\mathbf{P}_{\varepsilon}^{\beta}$
(cf. (\ref{eq:4.1-1})), $\mathcal{I}_{1}$ can be estimated by
\begin{equation}
\mathcal{I}_{1}\ge -C_{1}\log\left(\gamma\left(T,\beta\right)\right)\left\Vert w_{\beta}^{\varepsilon}\right\Vert _{L^{2}\left(\Omega\right)}^{2},\label{eq:I1}
\end{equation}
with the aid of H\"older's inequality.

Using the Young inequality and the structural property of the operator
$\mathbf{Q}_{\varepsilon}^{\beta}$ (cf. (\ref{eq:4.1-3})), $\mathcal{I}_{2}$
can be estimated by
\begin{equation}
\mathcal{I}_{2}\ge -C_{0}^{2}\gamma^{-2}\left(T,\beta\right)\left\Vert u\right\Vert _{\mathbb{W}}^{2}-\frac{1}{4}\left\Vert w_{\beta}^{\varepsilon}\right\Vert _{L^{2}\left(\Omega\right)}^{2}.\label{eq:I2}
\end{equation}

From now on, taking also into account the Lipschitz constant
$L_{F}$ and choosing an appropriate
Young inequality, we get the estimate of $\mathcal{I}_{3}$ as follows:
\begin{align}\label{eq:II3}
	\mathcal{I}_{3} & \ge -\frac{e^{2\rho_{\beta}\left(t-T\right)}}{8L_{F}^{2}}\left\Vert F\left(u_{\beta}^{\varepsilon}\right)-F\left(u\right)\right\Vert_{L^{2}\left(\Omega\right)}^{2}
	-2L_{F}^{2}\left\Vert w_{\beta}^{\varepsilon}\right\Vert _{\left[L^{2}\left(\Omega\right)\right]^{N}}^{2} \\
	& \ge -\left(\frac{1}{4} + 2 L_F^2\right)\left\Vert w_{\beta}^{\varepsilon}\right\Vert _{L^{2}\left(\Omega\right)}^{2}.\nonumber 
\end{align}

Plugging (\ref{eq:I1}), (\ref{eq:I2}) and (\ref{eq:II3}) 
into (\ref{eq:4.25}), and then integrating the resulting estimate
from $t$ to $T$ we obtain, after some rearrangement, that
\begin{align}
	\left\Vert w_{\beta}^{\varepsilon}\left(T\right)\right\Vert _{L^{2}\left(\Omega\right)}^{2}
	&+2\left(T-t\right)C_{0}^{2}\gamma^{-2}\left(T,\beta\right)\left\Vert u\right\Vert _{\mathbb{W}}^{2}\label{eq:4.32}\\
	& \ge\left\Vert w_{\beta}^{\varepsilon}\left(t\right)\right\Vert _{L^{2}\left(\Omega\right)}^{2}+2\underline{M}\int_{t}^{T}\left\Vert \nabla w_{\beta}^{\varepsilon}\left(s\right)\right\Vert _{\left[L^{2}\left(\Omega\right)\right]^{d}}^{2}ds,\nonumber 
\end{align}
by putting $\rho_{\beta}=C_{1}\log\left(\gamma\left(T,\beta\right)\right)+\frac{1}{2}+ 2 L_F^2>0$.

Note here that the existence of $u_{\beta}^{\varepsilon}\in L^{2}\left(0,T;H_{0}^{1}\left(\Omega\right)\right)$
has already been obtained in \cref{subsec:4.2}. Due to $\left(\text{A}_{4}\right)$ the first norm on the left-hand side of \eqref{eq:4.32} is bounded from above by $\varepsilon^{2}$. By the back-substitution $w_{\beta}^{\varepsilon}\left(x,t\right):=e^{\rho_{\beta}\left(t-T\right)}\left[u_{\beta}^{\varepsilon}\left(x,t\right)-u\left(x,t\right)\right]$ and the choice of $\rho_{\beta}$, we thus conclude that
\begin{align}
	& \left\Vert u_{\beta}^{\varepsilon}\left(\cdot,t\right)-u\left(\cdot,t\right)\right\Vert _{L^{2}\left(\Omega\right)}^{2}+2\underline{M}\int_{t}^{T}\left\Vert \nabla u_{\beta}^{\varepsilon}\left(\cdot,s\right)-\nabla u\left(\cdot,s\right)\right\Vert _{\left[L^{2}\left(\Omega\right)\right]^{d}}^{2}ds\label{eq:err1}\\
	& \le\gamma^{2C_{1}\left(T-t\right)}\left(T,\beta\right)\left(\varepsilon^{2}+2\left(T-t\right)C_{0}^{2}\gamma^{-2}\left(T,\beta\right)\left\Vert u\right\Vert _{C\left(\left[0,T\right];\mathbb{W}\right)}^{2}\right)e^{2\left(T-t\right)C_{2}},\nonumber 
\end{align}
where we have denoted by
\begin{equation}
C_{2}:=\frac{1}{2}+2L_F^2.\label{eq:CC2}
\end{equation}

Together with the $\varepsilon$-dependent blow-up rate of $\gamma$ in \eqref{blowup}, this ends the proof of the theorem.

\subsubsection{Proof of \cref{thm:Estimate2}}\label{subsec:5.3}
It is clear that in \cref{thm:Estimate1} the convergence does not hold at $t=0$. Taking a number $t^{\varepsilon}\in\left(0,T\right)$, we
prove that for each $\varepsilon>0$, there exists $t^{\varepsilon}>0$
such that $u_{\beta}^{\varepsilon}\left(x,t=t^{\varepsilon}\right)$
is a good approximation candidate of $u\left(x,t=0\right)$. Indeed,
if the source condition \eqref{eq:4.37-1} holds true, we get
\begin{align*}
	& \left\Vert u_{\beta}^{\varepsilon}\left(\cdot,t^{\varepsilon}\right)-u\left(\cdot,0\right)\right\Vert _{L^{2}\left(\Omega\right)}\\
	& \le\left\Vert u_{\beta}^{\varepsilon}\left(\cdot,t^{\varepsilon}\right)-u\left(\cdot,t^{\varepsilon}\right)\right\Vert _{L^{2}\left(\Omega\right)}+\left\Vert u\left(\cdot,t^{\varepsilon}\right)-u\left(\cdot,0\right)\right\Vert _{L^{2}\left(\Omega\right)}\\
	& \le\gamma^{-C_{1}t^{\varepsilon}}\left(T,\beta\right)\left(K+\sqrt{2T}C_{0}\gamma^{C_{1}T-1}\left(T,\beta\right)\left\Vert u\right\Vert _{C\left(\left[0,T\right];\mathbb{W}\right)}\right)e^{TC_{2}}\\
	& +t^{\varepsilon}\left\Vert u_{t}\right\Vert _{C\left(0,T;L^{2}\left(\Omega\right)\right)}.
\end{align*}

Observe that the error bound $\left\Vert u_{\beta}^{\varepsilon}\left(\cdot,t^{\varepsilon}\right)-u\left(\cdot,0\right)\right\Vert _{L^{2}\left(\Omega\right)}$ is essentially decided by the infimum of $\frac{1}{2}\left(\gamma^{-C_1 t^{\varepsilon}}(T,\beta) + t^{\varepsilon}\right)$ with respect to $t^{\varepsilon}>0$. We find that the term $\gamma^{-C_1 t^{\varepsilon}}(T,\beta)$ is decreasing and $t^{\varepsilon}$ obviously possesses a linear growth. Therefore, for every $\beta:=\beta\left(\varepsilon\right)>0$  there exists
a unique $t^{\varepsilon}\in\left(0,T\right)$ such that
\begin{equation}
\begin{cases}
{\displaystyle \lim_{\varepsilon\to0^{+}}t^{\varepsilon}=0},\\
t^{\varepsilon}=\gamma^{-C_{1}t^{\varepsilon}}\left(T,\beta\right),
\end{cases}\label{eq:4.39}
\end{equation}
and the second equation can be rewritten as
\begin{equation}
\frac{\log\left(t^{\varepsilon}\right)}{t^{\varepsilon}}=-C_{1}\log\left(\gamma\left(T,\beta\right)\right).\label{eq:4.37}
\end{equation}

Using the elementary inequality $\log\left(a\right)>-a^{-1}$ for all $a>0$,
it follows from (\ref{eq:4.37}) that
\[
t^{\varepsilon}<\sqrt{\frac{1}{C_{1}\log\left(\gamma\left(T,\beta\right)\right)}}.
\]

Henceforward, for $t^{\varepsilon}$ sufficiently small we complete the proof of the theorem.

\subsubsection{Proof of \cref{cor:5.3}}\label{subsec:5.4}
In this part, we rely on the Gagliardo-Nirenberg interpolation inequality for functions having zero trace to derive the error estimate. Essentially, it reads as
\begin{align}
	& \int_{t}^{T}\left\Vert u_{\beta}^{\varepsilon}\left(\cdot,s\right)-u\left(\cdot,s\right)\right\Vert _{L^{r}\left(\Omega\right)}^{2}ds\label{eq:4.41}\\
	& \le C_{\Omega}^{2}\left\Vert u_{\beta}^{\varepsilon}-u\right\Vert _{C\left(\left[t,T\right];L^{2}\left(\Omega\right)\right)}^{2\alpha}\int_{t}^{T}\left\Vert \nabla\left(u_{\beta}^{\varepsilon}-u\right)\left(\cdot,s\right)\right\Vert _{\left[L^{2}\left(\Omega\right)\right]^{d}}^{2\left(1-\alpha\right)}ds,\nonumber 
\end{align}
where $C_{\Omega}>0$ is a generic constant that only depends on the
geometry of $\Omega$, and the involved parameters should hold with:
$r>2$ and $0<\alpha<1$ satisfying
\[
\frac{1}{r}>\frac{d-2}{2d}\;\text{and }\frac{1}{r}=\frac{\alpha}{2}+\frac{\left(1-\alpha\right)\left(d-2\right)}{2d}.
\]
Note that (\ref{eq:4.41}) is available because of the
existence of $u_{\beta}^{\varepsilon}\in L^{2}\left(0,T;H_{0}^{1}\left(\Omega\right)\right)\cap L^{\infty}\left(0,T;L^{2}\left(\Omega\right)\right)$
leading to $C\left(\left[0,T\right];L^{2}\left(\Omega\right)\right)$;
see \cref{subsec:4.2}, and the compact
embedding $H^{1}\left(\Omega\right)\subset L^{r}\left(\Omega\right)$.
The special case of (\ref{eq:4.41}) in two and three-dimensional
versions ($d=2,3$) is the well known Ladyzhenskaya inequality.

Using H\"older's inequality we can write (\ref{eq:4.41}) as
\begin{align}
	& \int_{t}^{T}\left\Vert u_{\beta}^{\varepsilon}\left(\cdot,s\right)-u\left(\cdot,s\right)\right\Vert _{L^{r}\left(\Omega\right)}^{2}ds\label{eq:4.46}\\
	& \le C_{\Omega}^{2}\left(T-t\right)^{\alpha}\left\Vert u_{\beta}^{\varepsilon}-u\right\Vert _{C\left(\left[t,T\right];L^{2}\left(\Omega\right)\right)}^{2\alpha}\left(\int_{t}^{T}\left\Vert \nabla\left(u_{\beta}^{\varepsilon}-u\right)\left(\cdot,s\right)\right\Vert _{\left[L^{2}\left(\Omega\right)\right]^{d}}^{2}ds\right)^{1-\alpha}.\nonumber 
\end{align}

We remark that in \eqref{eq:4.46} we are only able to get the convergence until the near zero point of time, i.e. it merely holds for $0<t<T$. Accordingly, it is straightforward to obtain the rate in $L^{r}$ from \eqref{eq:4.38}. Thus, we complete the proof of the corollary. 

\section{Discussions}
\label{sec:conclusions}

\subsection{Some remarks on the system \eqref{Pb1}}\label{sec:remarks}

Having completed main results for the semi-linear case \eqref{Pb1-new}, it now suffices to provide some amendable remarks surrounding the general system \eqref{Pb1} and its regularization \eqref{reguproblem}.
\subsubsection*{Uniqueness result}
It is discernible that the regularized problem may have many solutions but those regularized solutions (if they exist) must converge to a unique true solution. Here we introduce collectively important steps, included in \cref{lem:3.1},   to prove the uniqueness result for the time-reversed system \eqref{Pb1} with the zero Dirichlet boundary condition. Then, from now onwards we will not come back to this issue in future publications for the regularization of this system. The technique we follow is mainly from \cite[Chapter 6]{Fried82}, which was used to
study the large-time behavior of solutions to a linear class of initial-boundary
value parabolic equations. Detailed proofs of the following results can be inspired from \cite{Khoa2018} for the observations in the semi-linear case \eqref{Pb1-new} with H\"older nonlinearities and the nonlinear Robin-type boundary condition.

Setting the function space 
\[
W_{T}\left(\Omega\right):=C\left(\left[0,T\right];H_{0}^{1}\left(\Omega\right)\cap W^{2,\infty}\left(\Omega\right)\right)\cap L^{\infty}\left(0,T;H^{2}\left(\Omega\right)\right)\cap C^{1}\left(0,T;L^{2}\left(\Omega\right)\right),
\]
we denote by $P_{T}\left(\Omega\right)$ the set of functions in $W_{T}\left(\Omega\right)$
such that they vanish on the boundary $\partial\Omega$ and at the
moments $t\in\left\{ 0,T\right\} $, i.e.
\[
P_{T}\left(\Omega\right):=\left\{ u\in W_{T}\left(\Omega\right):\left.u\right|_{\partial\Omega}=0,\left.u\right|_{t=T}=0,\left.u\right|_{t=0}=0\right\} .
\]
Then, for $\eta>0$ we set
\begin{align}
\lambda\left(t\right)=t-T-\eta.\label{lambdadef}
\end{align}
In what follows, this function plays a prime factor to prove the backward uniqueness result. According to solid proofs in \cite{Khoa2018}, it is also worth noting that \cref{lem:3.1} is essentially a Carleman estimate with the weight $\lambda^{-\frac{m}{k}}$; see \cite{Klibanov2015} for the observation in this spirit.

\begin{lemma}\label{lem:3.1}
	Assume the diffusion $a_{ij}(x,t,\cdot,\cdot)\in C^1(\overline{Q_T})$ for $1\le i,j\le N$ is such that it satisfies the strict ellipticity condition and the mapping $(\mathbf{p},\mathbf{q})\mapsto a\left(x,t;\mathbf{p};\mathbf{q}\right)$
	is sesquilinear for $\left(\mathbf{p},\mathbf{q}\right)\in [L^{2}\left(\Omega\right)]^{N}\times\left[L^{2}\left(\Omega\right)\right]^{Nd}$. For any $v\in [P_{T}\left(\Omega\right)]^{N}$,
	for any positive $m$ and any positive real $k$, one has
	\begin{align*}
	\left\Vert \lambda^{-\frac{m}{k}}\left(\nabla\cdot(a\left(v;\nabla v\right)\nabla v)-v_{t}\right)\right\Vert _{[L^{2}\left(Q_{T}\right)]^{N}}^{2}\nonumber \\
	\ge\frac{m}{k}\left\Vert \lambda^{-\frac{m}{k}-1}v\right\Vert _{[L^{2}\left(Q_{T}\right)]^{N}}^{2} & -D\left\Vert \lambda^{-\frac{m}{k}}\nabla v\right\Vert _{[L^{2}\left(Q_{T}\right)]^{Nd}}^{2},
	\end{align*}
	where $D$ depends only on the bounds of $\partial_{t}a$. Moreover, if
	$0<T\le\mu$ for $0<\mu\le\mu_{0}$ and $0<\eta\le\eta_{0}$ sufficiently
	small, there exists a positive $K$ independent of $m$ such that
	\begin{align}
	K\left\Vert \lambda^{-\frac{m}{k}}\left(\nabla\cdot\left(a\left(v;\nabla v\right)\nabla v\right)-v_{t}\right)\right\Vert _{\left[L^{2}\left(Q_{T}\right)\right]^{N}}^{2}\label{eq:3.10} \\
	\ge\left\Vert \lambda^{-\frac{m}{k}-1}v\right\Vert _{\left[L^{2}\left(Q_{T}\right)\right]^{N}}^{2} & +\frac{1}{2}\left\Vert \lambda^{-\frac{m}{k}}\nabla v\right\Vert _{\left[L^{2}\left(Q_{T}\right)\right]^{Nd}}^{2},
	\nonumber
	\end{align}
	for $m$ sufficiently large.
\end{lemma}

Let $u$ and $v$ be the two solutions of the backward problem \eqref{Pb1}-\eqref{terminal}
in $[W_{T}\left(\Omega\right)]^{N}$. The difference
system for $w = u - v$ reads as
\begin{align}
w_{t}+\nabla\cdot\left(-a\left(x,t;w;\nabla w\right)\nabla w\right) & =F\left(x,t;u;\nabla u\right)-F\left(x,t;v;\nabla v\right)\label{eq:3.1} \\
& +\nabla\cdot\left(a\left(x,t;u;\nabla u\right)\nabla u\right)-\nabla\cdot\left(a\left(x,t;v;\nabla v\right)\nabla v\right)\nonumber \\
& -\nabla\cdot\left(a\left(x,t;w;\nabla w\right)\nabla w\right),\nonumber
\end{align}
endowed with the zero Dirichlet boundary condition and the zero terminal
condition.

Under the assumptions that $a,F$ are Lipschitz-continous with respect to the nonlinear arguments $\mathbf{p},\mathbf{q}$ and that $a$ satisfies the strict ellipticity condition, we can find a positive constant
$C$ such that from (\ref{eq:3.1}) the following differential inequality holds
\begin{equation}
\left|w_{t}+\nabla\cdot\left(-a\left(x,t;w;\nabla w\right)\nabla w\right)\right|^{2}\le C\left(\left|w\right|^{2}+\left|\nabla w\right|^{2}\right).\label{eq:3.3}
\end{equation}

Observe that $w\in [P_{T}\left(\Omega\right)]^{N}$, we can obtain the uniqueness result in $[P_{T}\left(\Omega\right)]^{N}$ for \eqref{Pb1} by using \eqref{eq:3.10}, \eqref{eq:3.3} and by choosing appropriately small values of $\mu_{0}$ and $\eta_{0}$.

\subsubsection*{Nonlocal diffusion} 
We could meliorate the existence result (cf. \cref{existencetheorem}) when the diffusion $a$ in the system \eqref{Pb1} is of the following physical types:
\begin{itemize}
	\item $a=a(x,t)$ typically accounting for the anisotropic diffusion and possibly taxis processes;
	\item $a=a\left(t;u\right)=\max\left\{ \theta_{0},\theta_{1}+\left|\int_{\Omega}u\left(x,t\right)dx\right|\right\} +\theta_{2}$ for some $\theta_{0},\theta_{1},\theta_{2}>0$. The diffusion in this form is controlled by the local movements of species involved in the evolution equation (see e.g. \cite{AADF16,TAKL17} for the concrete biological motivation of this equation);
	\item $a=a\left(t;\left\Vert \nabla u\right\Vert _{L^{2}\left(\Omega\right)}^{2}\right)=\theta_{3}+\int_{\Omega}\left|\nabla u\right|^{2}dx$ for some $\theta_3 >0$  indicating a Kirchhoff-type diffusion model for e.g. flows through
	porous media.
\end{itemize}
Using the same argument in \cref{existencetheorem}, it is worth mentioning that the convergence results obtained in \eqref{eq:4.13} and \eqref{eq:4.16-3} are sufficient to passing to the limit in the diffusion term involving the aforementioned forms. Consequently, the existence result remains true in these cases for any spatial dimensions $d$. However, this technique is not valid for the $p$-Laplacian equation inspired from the power-law type of Ohm's law in conductivity of electricity, which reads as
\begin{equation}
u_{t}-\nabla\cdot\left(\left|\nabla u\right|^{p-2}\nabla u\right)=F\left(u\right) \quad \text{for } p \ge 2,\label{eq:Ohm}
\end{equation}
due to the failure of passage to the limit. When $d=1$, there is a possibility of proving this solvability by the embedding $H_0^1(\Omega)\subset L^{\infty}(\Omega)$. 

Since we use the boundedness of the diffusion term as a key point in the convergence analysis, a slight improvement of  \cref{thm:Estimate1} and \cref{thm:Estimate2} can be obtained when $a$ is dependent of the gradient. In fact, assuming the source condition (compared to \eqref{eq:4.37-1})
\begin{align}
u\in C\left(\left[0,T\right];\mathbb{W}\right)\cap L^{\infty}(0,T;W^{1,\infty}(\Omega))\cap C^{1}\left(0,T;L^{2}\left(\Omega\right)\right),\label{sourcecond}
\end{align}
one could suppose that $\underline{M}\ge \eta \left\Vert \nabla u\right\Vert _{L^{\infty}\left(Q_T\right)}$ for some $\eta >0$ sufficiently small, somewhat similar to the concept of large diffusion  in terms of $A$, to gain similar error bounds. Technically, the reason behind this assumption is to preserve the positivity of the gradient term in \eqref{eq:4.32}. In some physical problems, the small diffusion $a$ would fit this circumstance because $\overline{M}$ now can be taken sufficiently large and then choosing $\underline{M}$ large is possible. 

\subsubsection*{Locally Lipschitz-continuous nonlinearities}
From now on, we extend the convergence analysis when the source term $F$ locally depends on $u$ and $\nabla u$. In this scenario, we need  the estimate \eqref{eq:2.1} for the cut-off function $F_{\ell}$ introduced in \cref{remarkphu}. Essentially, there are two main  difficulties in the proofs.
\begin{itemize}
	\item When exploring the difference equation in proof of \cref{thm:Estimate1} we confront with the difference term $F_{\ell^{\varepsilon}}(u_{\beta}^{\varepsilon};\nabla u_{\beta}^{\varepsilon})-F(u;\nabla u)$. Thus, estimating $\mathcal{I}_{3}$ in \eqref{eq:4.25} would be problematic.
	\item This moment the constant $C_2$ in \eqref{eq:err1} and given by \eqref{eq:CC2} would depend on $\ell^{\varepsilon}$. Observe that the behavior of $\ell^{\varepsilon}$ should be increasing (when $\varepsilon\to 0$) as it approximates the source function $F$ in \eqref{eq:2.1-1}. Therefore, this parameter must be formulated in a clear manner to ensure the convergence of our QR scheme.
\end{itemize}

These issues are really needed to elucidate because, as particularly mentioned in \cref{subsec:background}, the local Lipschitz continuity of $F$ is encountered in most of the significant equations in real-life applications. Here we sketch out some essential ideas that we can adapt to the proof of \cref{thm:Estimate1}. Note that here we need the aid of the source condition \eqref{sourcecond}.

At first, we choose the cut-off parameter $\ell^{\varepsilon}>0$
such that
\begin{align}\label{choice1}
\ell^{\varepsilon}\ge\left\Vert u\right\Vert _{L^{\infty}\left(0,T;W^{1,\infty}\left(\Omega\right)\right)}.
\end{align}
This way we solve the first issue because $F_{\ell^{\varepsilon}}\left(x,t;u;\nabla u\right)=F\left(x,t;u;\nabla u\right)$; cf.
(\ref{eq:2.1-1}).

Taking into account the Lipschitz constant
$L_{F}\left(\ell^{\varepsilon}\right)>0$ and choosing an appropriate
Young inequality, we get the estimate of $\mathcal{I}_{3}$ as follows:
\begin{align}\label{eq:I3}
\mathcal{I}_{3} & \ge -\frac{e^{2\rho_{\beta}\left(t-T\right)}\underline{M}}{8L_{F}^{2}\left(\ell^{\varepsilon}\right)}\left\Vert F_{\ell^{\varepsilon}}\left(u_{\beta}^{\varepsilon};\nabla u_{\beta}^{\varepsilon}\right)-F_{\ell^{\varepsilon}}\left(u;\nabla u\right)\right\Vert _{L^{2}\left(\Omega\right)}^{2}
-\frac{2L_{F}^{2}\left(\ell^{\varepsilon}\right)}{\underline{M}}\left\Vert w_{\beta}^{\varepsilon}\right\Vert _{L^{2}\left(\Omega\right)}^{2} \\
& \ge -\frac{\underline{M}}{4}\left(\left\Vert w_{\beta}^{\varepsilon}\right\Vert _{L^{2}\left(\Omega\right)}^{2}+\left\Vert \nabla w_{\beta}^{\varepsilon}\right\Vert _{\left[L^{2}\left(\Omega\right)\right]^{d}}^{2}\right)
-\frac{L_{F}^{2}\left(\ell^{\varepsilon}\right)}{\underline{M}}\left\Vert w_{\beta}^{\varepsilon}\right\Vert _{L^{2}\left(\Omega\right)}^{2}.\nonumber
\end{align}
Henceforward, \eqref{eq:4.32} remains the same when we put
$\rho_{\beta}= C_1 \log(\gamma(T,\beta)) + \frac{\underline{M}+1}{4}
+\frac{L_F^2 (\ell^{\varepsilon})}{\underline{M}}$. With this choice, the constant $C_2$ in \eqref{eq:CC2} is $\varepsilon$-dependent and given by
\begin{equation}
C_{2}\left(\ell^{\varepsilon}\right):=\frac{\underline{M}+1}{4}+\frac{L_{F}^{2}\left(\ell^{\varepsilon}\right)}{\underline{M}}.\label{eq:C2}
\end{equation}

Now observe \eqref{eq:err1} with this new $C_{2}$ in \eqref{eq:C2} and have in mind that the error estimate at $t=0$ (cf. \cref{thm:Estimate2}) is of the order $\mathcal{O}\left(\log^{-\frac{1}{2}}(\gamma(T,\beta))\right)$. We only need to find a fine control of the term  $e^{\frac{\left(T-t\right)L_{F}^{2}\left(\ell^{\varepsilon}\right)}{\underline{M}}}$ in such a way that its growth does not ruin the logarithmic rate of convergence. To do so, our strategy is the following: We take 
\begin{align}\label{choice2}
\varrho:=\varrho\left(\beta\right)=\sqrt{\frac{\underline{M}}{T}\log\left(\log^{\kappa}\left(\gamma\left(T,\beta\right)\right)\right)}>0,
\end{align}
for some $\varepsilon$-independent constant $\kappa>0$ being selected later. Then, we have
\begin{equation}
\lim_{\varepsilon\to0^{+}}\varrho\left(\beta\right)=\infty.\label{eq:4.33}
\end{equation}
If we choose $\Lambda^{\beta}:=\sup L_{F}^{-1}\left\{ \left(-\infty,\varrho\left(\beta\right)\right]\right\} $,
then $L_{F}\left(\Lambda^{\beta}\right)=\varrho\left(\beta\right)$
and we also obtain
\begin{equation}
e^{\frac{\left(T-t\right)L_{F}^{2}\left(\Lambda^{\beta}\right)}{\underline{M}}}\le\log^{\kappa}\left(\gamma\left(T,\beta\right)\right).\label{eq:4.34}
\end{equation}

Note also that by (\ref{eq:4.33}), $L_{F}^{-1}\left\{ \left(-\infty,\varrho\left(\beta\right)\right]\right\} \ne0$
and $\Lambda^{\beta}\in\left(0,\infty\right)$ is well-defined. Moreover,
we can prove that  $\lim_{\varepsilon\to0^{+}}\Lambda^{\beta}=\infty$.
Indeed, we suppose that there exists $C>0$ such that $\Lambda^{\beta}\le C$
for $\beta$ near the zero point. Since $L_{F}$ is non-decreasing with respect to $\ell^{\varepsilon}$,
it holds $L_{F}\left(C\right)\ge L_{F}\left(\Lambda^{\beta}\right)=\varrho\left(\beta\right)$,
which contradicts the fact already known (\ref{eq:4.33}). Now, for $\ell^{\varepsilon}\in\left(0,\Lambda^{\beta}\right]$ we
deduce that
\[
e^{\frac{\left(T-t\right)L_{F}^{2}\left(\ell^{\varepsilon}\right)}{\underline{M}}}\le\log^{\kappa}\left(\gamma\left(T,\beta\right)\right),
\]
resulted from (\ref{eq:4.34}). This also indicates that we have identified
a fine upper bound of the $\ell^{\varepsilon}$-dependent Lipschitz
constant $L_{F}$, and the error estimate \eqref{eq:err1} now becomes
\begin{align}
& \left\Vert u_{\beta}^{\varepsilon}\left(\cdot,t\right)-u\left(\cdot,t\right)\right\Vert^{2} _{L^{2}\left(\Omega\right)}+2\underline{M}\int_{t}^{T}\left\Vert \nabla u_{\beta}^{\varepsilon}\left(\cdot,s\right)-\nabla u\left(\cdot,s\right)\right\Vert^{2} _{\left[L^{2}\left(\Omega\right)\right]^{d}}ds\label{eq:4.38}\\
& \le\log^{2\kappa}\left(\gamma\left(T,\beta\right)\right)\gamma^{2C_{1}\left(T-t\right)}\left(T,\beta\right)\left(\varepsilon^{2}+2TC_{0}^{2}\gamma^{-2}\left(T,\beta\right)\left\Vert u\right\Vert^{2} _{C\left(\left[0,T\right];\mathbb{W}\right)}\right)e^{2TC_{3}},\nonumber 
\end{align}
where $C_{3}:=\frac{\underline{M}+1}{4}$ is no longer dependent of $\ell^{\varepsilon}$.

Similar to proof of \cref{thm:Estimate2}, we inherit from \eqref{eq:4.38} to gain the error estimate at $t=0$ with the order $\mathcal{O}\left(\log^{\kappa-\frac{1}{2}}(\gamma(T,\beta))\right)$. Hence, together with \eqref{eq:4.38} we choose $\kappa:=\kappa(t)=\min\left\{C_1 t,\frac{1}{2}\right\}>0$ to complete the convergence analysis in this case. On top of this, the choice of the cut-off parameter can be summarized by \eqref{choice1} and \eqref{choice2}, working with sufficiently small values of $\varepsilon$.

\subsubsection*{No-flux boundary condition}
Since our problem \eqref{Pb1}-\eqref{terminal} is also present in population dynamics, the zero Neumann condition should be analyzed. In this case, we associate the regularized problem \eqref{regune} with the boundary condition $-a\nabla u^{\varepsilon} \cdot \text{n} = 0$, taking the place of the zero Dirichlet boundary condition in \eqref{reguproblem-1}. Under this setting, the techniques used in the proofs of our main results can be applied in the same manner, focusing on the same structure of the weak formulation we have in \eqref{eq:4.2} (where the test function $\psi$ now belongs to the closed subspace of $H^1(\Omega)$ that satisfies the zero Neumann boundary condition) and the key equation \eqref{eq:4.25} for the convergence analysis. Accordingly, the rates of convergence derived in \cref{thm:Estimate1} and \cref{thm:Estimate2} remain unchanged. Moreover, the strong convergence on the boundary is confirmed for $0<t<T$ by the following trace inequality: 
\begin{align*}
& \int_{t}^{T}\left\Vert u_{\beta}^{\varepsilon}\left(\cdot,s\right)-u\left(\cdot,s\right)\right\Vert _{[L^{2}\left(\partial\Omega\right)]^{N}}^{2}ds\\
& \le C_{\Omega}\left(\left\Vert u_{\beta}^{\varepsilon}-u\right\Vert _{[C\left(\left[t,T\right];L^{2}\left(\Omega\right)\right)]^{N}}^{2}+\int_{t}^{T}\left\Vert \nabla\left(u_{\beta}^{\varepsilon}-u\right)\left(\cdot,s\right)\right\Vert _{[L^{2}\left(\Omega\right)]^{Nd}}^{2}ds\right),
\end{align*}
which yields the same rate as in \cref{thm:Estimate1}.

\subsection{Possible future generalizations of above results}
\subsubsection*{Gevrey class} It is worth noting that the property \eqref{prop} remains true up to a compact Riemannian manifold, which is generally called the Sturm-Liouville decomposition. As a prominent example, the standard eigen-elements for a $d$-torus $\mathbb{T}^{d}=\mathbb{R}^{d}/\left(2\pi\mathbb{Z}\right)^{d}$ are
\[
\phi_{p}\left(x\right)=\prod_{j=1}^{d}e^{2\pi ip_{j}x_{j}},\quad\mu_{p}=\sum_{j=1}^{d}\left(2\pi p_{j}\right)^{2},\quad p_{j}\in\mathbb{N},1\le j\le d,\; i=\sqrt{-1}.
\]
In this scenario, Gevrey classes are popular in micro-local analysis for the propagation of wavefront set and in the study of analytic regularity for nonlinear evolution equations with periodic boundary data. A famous result of the Gevrey solvability for nonlinear analytic parabolic equations is recalled in an example of \cref{appendix1}. Here, our discussions focus on the preasymptotic error bounds for \emph{approximation numbers} of periodic Gevrey-type spaces of analytic functions with connection to the Galerkin method.

For $0<\alpha,p,q<\infty$, we denote by $\mathbb{G}_{\alpha}^{p,q}(\mathbb{T}^{d})$ the Gevrey space that consists of all functions in $C^{\infty}(\mathbb{T}^{d})$ and satisfies
\[
\left\Vert u\right\Vert _{\mathbb{G}_{\alpha}^{p,q}(\mathbb{T}^{d})} := \left(\sum_{k\in \mathbb{Z}^{d}}\text{exp}\left(2\alpha \left\Vert k\right\Vert^{q}_{p} \right)\hat{u}_{k}\right)^{1/2} < \infty,
\]
where $\hat{u}_{k}$  denotes the Fourier coefficient of $u$ with respect to the frequency vector $k\in \mathbb{Z}^{d}$. By this definition, the norm $\left\Vert e^{\overline{M}T\left(-\Delta\right)}u\right\Vert _{L^{2}\left(\mathbb{T}^{d}\right)}$ in \cref{exampleope} is essentially $\left\Vert u\right\Vert_{\mathbb{G}_{\overline{M}T}^{2,2}(\mathbb{T}^{d})}$. For $q\in (0,1)$, this space is the classical Gevrey classes that contain non-analytic functions, whilst for $q \ge 1$ all functions are real-analytic therein.

In approximation theory for Hilbert spaces, approximation numbers represent the worst-case error obtained when approximating a class of functions by projecting them onto the optimal finite-dimensional subspace. The basic reason lies in the information-based complexity  that requires the rank $n\in \mathbb{N}$ of the optimal projection operator is sufficiently large ($n > 2^{d}$) to gain the classical error bounds, which is not substantially practical for high dimensions. Therefore, approximation numbers can be an excellent candidate to handle this context. In a nutshell, the connection between such approximation numbers and Galerkin schemes for a classical variational problem, where a parabolic problem can  be involved, is clearly present in \cite[Subsection 1.5]{KMU16} with references cited therein for a background of Gevrey classes.

\begin{definition}[Approximation numbers]
	Let $X$ and $Y$ be two Banach spaces. The norm of an operator $\mathcal{A}: X \to Y$ is denoted by $\left\Vert \mathcal{A}\right\Vert_{X\to Y}$. The $n$th approximation number ($n\in\mathbb{N}$) of an operator $\mathcal{T}: X \to Y$ is defined by
	\[
	{a}_{n}\left(\mathcal{T}:X\to Y\right):=\inf_{\text{rank}\left(\mathcal{A}\right)<n}\left\Vert \mathcal{T}-\mathcal{A}\right\Vert _{X\to Y}.
	\]
\end{definition}

Taking into account the Gevrey space that have been mentioned above,  the approximation numbers of the embedding $\text{Id}:\mathbb{G}_{\overline{M}T}^{2,2}(\mathbb{T}^{d}) \to L^2(\mathbb{T}^{d})$ are bounded by
\[
n^{-\frac{c_{1}\overline{M}T}{\log_{2}\left(1+d/\log_{2}\left(n\right)\right)}}\le a_{n}\left(\text{Id}:\mathbb{G}_{\overline{M}T}^{2,2}\left(\mathbb{T}^{d}\right)\to L^{2}\left(\mathbb{T}^{d}\right)\right)\le n^{-\frac{c_{2}\overline{M}T}{\log_{2}\left(1+d\right)}},
\]
for $d\le n \le 2^d$ and $c_1,c_2 >0$.
This rigorous estimate is almost identical to the preasymptotic estimate for approximation numbers of the classical embeddings $\text{Id}: H^1\left(\mathbb{T}^{d}\right) \to L^2\left(\mathbb{T}^{d}\right)$, albeit $\mathbb{G}_{\overline{M}T}^{2,2}(\mathbb{T}^{d})$ obviously contains smoother functions than $H^1\left(\mathbb{T}^{d}\right)$. On the other hand,  the approximation numbers for the embedding $\text{Id}: \mathbb{G}_{\overline{M}T}^{2,2}(\mathbb{T}^{d}) \to H^1\left(\mathbb{T}^{d}\right)$ are asymptotically identical when $1\le n\le d$, while for the embedding $\text{Id}: W^{1,\infty}(\mathbb{T}^{d}) \to L^2(\mathbb{T}^{d})$, they are completely identical whenever $n \le 2^d$.

Eventually, all these numbers indicate that there is a possibility to choose a combination of an optimal dimensional subspace and a linear finite element algorithm such that an approximate numerical solution by Galerkin methods is a good candidate in $L^2$ and $H^1$ for the true solution satisfying \eqref{eq:dku}. Consequently, the worst-case \emph{a priori} error for the (low) $n$-dimensional subspace in this context behaves like that of the standard finite element methods (FEMs).

\subsection{Concluding remarks}
We have extended a modified quasi-reversibility (QR) method for backward quasi-linear
parabolic systems with noise. Several rates of convergence have been derived, especially the rigorous
error estimates in $L^{r}\left(\Omega\right)$ ($r\ge2$) and $H^{1}\left(\Omega\right)$, albeit many open questions remain unsolved. Although the spectral method that takes into account Duhamel’s principle is not used, settings for filter regularized operators still  rely on existence of the space $\mathbb{W}$, which usually plays a role as a class of Gevrey spaces in the existing trend of regularization for time-reversed nonlinear parabolic equations.

Our present contribution gives rise to some further interesting questions.
Recently, we have only done with several error estimates which indicate
the strong convergence of the regularization scheme. This typical
convergence is not expected to be applied in the stochastic setting,
but it can be designed to obtain an approximate solution
in the FEM framework. In this sense, our theoretical
analysis can be a key ingredient to establish regularized multiscale
FEM schemes which deal with models in certain complex domains because
spatial environments where population densities take place are usually
not nice (e.g. porous media). Other
open perspectives include the effective iterative QR method
and also the presence of the Robin-type boundary condition describing
e.g. the surface reaction in more complex scenarios.

\appendix

\section{Applications to existing models}\label{appendix1}
Here, we examine four types of backward problems arising in many physical applications to show the applicability of our theoretical
analysis. In order to show existing arguments on the \emph{a priori} information \eqref{sourcecond} where $\mathbb{W}$ stands for a class of Gevrey spaces demonstrated in \cref{exampleope}, we specify below the possible regularity assumptions for different models chosen from simple to complex, based on the analysis of the forward models. Note that $1 \le d \le 3$ are only considered due to the practical meaning. 

\subsection{Fisher-Burger equation} In a finite interval $[0,l]$ with periodic boundary condition, one concerns the following equation:
\[
u_{t}+Cuu_{x}=Du_{xx}+Bu\left(1-u\right),
\]
with $B,C,D$ being positive constants, for simplicity.

This problem is performed as a combination between the classical Fisher and Burger equations. Here we can further consider the real analytic cases with respect to $u$ of the nonlinear $F$ which imply several types of modelling interactions between particles. We know that in \cite{FT07} the authors proved the local weak solvability of the forward problem. In this sense, if the initial  condition is sufficiently smooth, viz. $u_{0}\in H^{1}(\Omega)$, then we obtain a unique solution $u\in \mathbb{G}^{2,2}_{t}$ for any $t\in [0,T^{*}]$ with $T^{*}$ sufficiently small. This not only verifies that the Gevrey regularity on the true solution could be valid in some certain models, but also agrees with the mild restriction of time in the convergence results.

\subsection{$p$-Laplacian equation} In a bounded domain with a H\"older boundary, we take into account the equation \eqref{eq:Ohm} with the zero Dirichlet boundary condition. Cf. \cite{Na17}, we can obtain the classical solution in $L^{\infty}(0,T;W_{0}^{1,\infty}(\Omega))$ when $u_{0}\in W^{1,\infty}(\Omega)$. Together with the Fisher-Burger equation, we remark that these forward problems have interesting phenomena including e.g. profiles of extinction and blow-up in finite time, the instantaneous shrinking of the support from the diffusion coefficient. Depending on the situation one may need appropriate choices of the auxiliary parameter $\rho_{\varepsilon}$ involved in the regularized problem to keep track of the arisen phenomena. Therefore, rigorous analysis of the regularized problem \eqref{reguproblem}-\eqref{reguproblem-1} will be considered in the forthcoming works.

\subsection{Gray-Scott-Klausmeier model} Based on the one-dimensional setting with $\Omega = \mathbb{R}$ in \cite{MRS14}, we set $u = (u_1 , u_2)$ with $u_1 > 0$ to guarantee the positive-definite diffusion $a(u)$. Then the closed-form nonlinearities are
\[
a\left(u\right)=\begin{pmatrix}2u_{1} & 0\\
0 & D
\end{pmatrix},\quad F\left(u,u_{x}\right)=\begin{pmatrix}Cu_{1x}+A\left(1-u_{1}\right)-u_{1}u_{2}^{2}\\
-Bu_{2}+u_{1}u_{2}^{2}
\end{pmatrix},
\]
where the involved parameters $A,B,C,D$ are positive.

This model describes the interaction between water $u_1$ and plant biomass $u_2$ in semiarid landscapes. The local well-posedness in $H^2(\mathbb{R})$ (cf. \cite[Theorem 2.2]{MRS14}) enjoys the possibility of taking $W^{1,\infty}$ in \eqref{sourcecond} due to the embedding $W^{1,1}(\mathbb{R}) \subset L^{\infty}(\mathbb{R})$.

\subsection{Shigesada-Kawasaki-Teramoto model}
In a three-dimensional setting with no-flux boundary condition, we consider
\[
a\left(u\right)=\begin{pmatrix}a_{10}+2a_{11}u_{1}+a_{12}u_{2} & a_{12}u_{1}\\
a_{21}u_{2} & a_{20}+2a_{22}u_{2}+a_{21}u_{1}
\end{pmatrix},
\]
where the non-negative coefficients $a_{ij}$ satisfy $8a_{11} \ge a_{12}$ and $8a_{22} \ge a_{21}$ to fulfill the positive-definiteness of diffusion. The source term is taken as the Lotka-Volterra functions, which reads as
\[
F\left(u\right)=\begin{pmatrix}\left(b_{10}-b_{11}u_{1}-b_{12}u_{2}\right)u_{1}\\
\left(b_{20}-b_{21}u_{1}-b_{22}u_{2}\right)u_{2}
\end{pmatrix},
\]
where the coefficients $b_{ij}$ are non-negative.

This famous model plays a vital role in population dynamics for multi-species systems in which self- and cross-diffusion effects are participated. An included example is the Keller-Segel model for cell aggregation, structured by setting $a_{10}=a_{20}=1$, $a_{11}=a_{12}=a_{21}=a_{22}=0$, $a_{12}= -1$ with $F(u) = \left(0,u_{1}- u_{2}\right)^{T}$. It is important to see that $a$ does not need to be symmetric in this context. Concerning the forward problem, the existence of bounded weak solution, i.e. $u_{i}\in L^{\infty}(0,T;L^{\infty}(\Omega))$, is proven in \cite{Gali12} if the initial data $u_{i}^{0} \in L^{\infty}(\Omega)$ for $i=1,2$. Moreover, if $\nabla u_{i} \in L^{\infty}(0,T;L^{\infty}(\Omega))$, the uniqueness result is obtained. Essentially, observe that we can adapt the \emph{a priori} argument $L^{\infty}(0,T;W^{1,\infty}(\Omega))$ in \eqref{sourcecond} to this model.

\section*{Acknowledgments}
The authors would like to thank the anonymous referee(s) and the handling editor(s) for fruitful  comments  through  the significant improvement  of  this  paper. V.A.K acknowledges Prof. Daniel Lesnic (Leeds, UK) and Prof. Thorsten Hohage (G\"ottingen, Germany) for recent supports to his research inception.

\bibliographystyle{siamplain}
\bibliography{mybib}
\end{document}

%% file: ex_shared.tex
\usepackage{amsfonts}
\usepackage{graphicx}
\usepackage{epstopdf}
\usepackage{algorithmic}
\usepackage{amssymb}
\usepackage{color}

\usepackage[utf8]{inputenc}
\usepackage[english]{babel}
\newtheorem{remark}[theorem]{Remark}
\newtheorem{example}[theorem]{Example}

\ifpdf
  \DeclareGraphicsExtensions{.eps,.pdf,.png,.jpg}
\else
  \DeclareGraphicsExtensions{.eps}
\fi

\numberwithin{theorem}{section}

\newcommand{\TheTitle}{Analysis of a quasi-reversibility method for a terminal value quasi-linear parabolic problem with measurements} 
\newcommand{\TheAuthors}{N.H. Tuan, V.A. Khoa, and V.V. Au}

\headers{A quasi-reversibility method for quasi-linear parabolic problems}{\TheAuthors}

\title{{\TheTitle}\thanks{Submitted to the editors DATE.}
\funding{The work of the first author and the third author was supported by the Institute for Computational Science and Technology Ho Chi Minh City under the project named ``Some ill-posed problems for partial differential equations''. The work of the second author was supported by a postdoctoral fellowship of the Research Foundation-Flanders (FWO).}}

\author{
  Nguyen Huy Tuan\thanks{Institute for Computational Science and Technology Ho Chi Minh City, Vietnam, and Faculty of Mathematics and Computer Science, University of Science, Vietnam
  	National University, 227 Nguyen Van Cu, District 5, Ho Chi Minh City, Vietnam
    (\email{nhtuan@hcmus.edu.vn}).}
  \and
  Vo Anh Khoa\thanks{Author for correspondence. Institute for Numerical and Applied Mathematics, University of Goettingen, 37083 Goettingen, Germany, and Faculty of Sciences, Hasselt University, Campus Diepenbeek, Agoralaan Building D, BE3590 Diepenbeek, Belgium (\email{vakhoa.hcmus@gmail.com}).}
  \and
  Vo Van Au\thanks{Institute for Computational Science and Technology Ho Chi Minh City, Vietnam, and Faculty of General Sciences, Can Tho University of Technology, Can Tho City, Vietnam (\email{vvau@ctuet.edu.vn}).}
}

\usepackage{amsopn}


%% file: ex_article.bbl
\begin{thebibliography}{10}

\bibitem{AADF16}
{\sc R.~Almeida, S.~Antontsev, J.~Duque, and J.~Ferreira}, {\em A
  reaction-diffusion model for the non-local coupled system: existence,
  uniqueness, long-time behaviour and localization properties of solutions},
  IMA Journal of Applied Mathematics, 81 (2016), pp.~344--364.

\bibitem{BE89}
{\sc J.~Bebernes and D.~Eberly}, {\em Mathematical Problems from Combustion
  Theory}, vol.~83 of {A}pplied {M}athematical {S}ciences, Springer-Verlag,
  1989.

\bibitem{BS04}
{\sc G.~Bluman and V.~Shtelen}, {\em Nonlocal transformations of {K}olmogorov
  equations into the backward heat equation}, Journal of Mathematical Analysis
  and Applications, 291 (2004), pp.~419--437.

\bibitem{BD09}
{\sc L.~Bourgeois and J.~Darde}, {\em About stability and regularization of
  ill-posed elliptic {C}auchy problems: the case of {L}ipschitz domains},
  Applicable Analysis, 89 (2009), pp.~1745--1768.

\bibitem{BD10}
{\sc L.~Bourgeois and J.~Darde}, {\em A duality-based method of
  quasi-reversibility to solve the {C}auchy problem in the presence of noisy
  data}, Inverse Problems, 26 (2010).
\newblock 095016.

\bibitem{CGT17}
{\sc V.~Calvez, L.~Gosse, and M.~Twarogowska}, {\em Travelling chemotactic
  aggregates at mesoscopic scale and bistability}, SIAM Journal on Applied
  Mathematics, 77 (2017), pp.~2224--2249.

\bibitem{CKP09}
{\sc H.~Cao, M.~V. Klibanov, and S.~V. Pereverzev}, {\em A {C}arleman estimate
  and the balancing principle in the quasi-reversibility method for solving the
  {C}auchy problem for the {L}aplace equation}, Inverse Problems, 25 (2009).
\newblock 035005.

\bibitem{Cara76}
{\sc A.~Carasso}, {\em Error bounds in the final value problem for the heat
  equation}, SIAM Journal on Mathematical Analysis, 7 (1976), pp.~195--199.

\bibitem{Cara77}
{\sc A.~Carasso}, {\em Computing small solutions of {B}urgers' equation
  backwards in time}, Journal of Mathematical Analysis and Applications, 59
  (1977), pp.~169--209.

\bibitem{Cara94}
{\sc A.~S. Carasso}, {\em Overcoming {H}\"older continuity in ill-posed
  continuation problems}, SIAM Journal on Numerical Analysis, 31 (1994),
  pp.~1535--1557.

\bibitem{Cara13}
{\sc A.~S. Carasso}, {\em Reconstructing the past from imprecise knowledge of
  the present: {E}ffective non-uniqueness in solving parabolic equations
  backward in time}, Mathematical Methods in the Applied Sciences, 36 (2013),
  pp.~249--261.

\bibitem{Cara78}
{\sc A.~S. Carasso, J.~G. Sanderson, and J.~M. Hyman}, {\em Digital removal of
  random media image degradations by solving the diffusion equation backwards
  in time}, SIAM Journal on Numerical Analysis, 15 (1978), pp.~344--367.

\bibitem{CKK16}
{\sc R.~Cherniha, J.~R. King, and S.~Kovalenko}, {\em Lie symmetry properties
  of nonlinear reaction-diffusion equations with gradient-dependent
  diffusivity}, Communications in Nonlinear Science and Numerical Simulation,
  36 (2016), pp.~98--108.

\bibitem{CQ04}
{\sc M.~Chipot and P.~Quittner}, {\em Handbook of {D}ifferential {E}quations:
  {S}tationary {P}artial {D}ifferential {E}quations}, vol.~1, Elsevier, 2004.

\bibitem{CD00}
{\sc J.~Cholewa and T.~Dlotko}, {\em Global {A}ttractors in {A}bstract
  {P}arabolic {P}roblems}, Cambridge University Press, Cambridge, UK, 2000.

\bibitem{CK07}
{\sc C.~Clason and M.~V. Klibanov}, {\em The quasi-reversibility method for
  thermoacoustic tomography in a heterogeneous medium}, SIAM Journal on
  Scientific Computing, 30 (2007), pp.~1--23.

\bibitem{DNKV17}
{\sc V.~N. Doan, H.~T. Nguyen, V.~A. Khoa, and V.~A. Vo}, {\em A note on the
  derivation of filter regularization operators for nonlinear evolution
  equations}, Applicable Analysis, 97 (2018), pp.~3--12.

\bibitem{DAAF16}
{\sc J.~C.~M. Duque, R.~Almeida, S.~N. Antontsev, and J.~Ferreira}, {\em The
  {E}uler-{G}alerkin finite element method for a nonlocal coupled system of
  reaction-diffusion type}, Journal of Computational and Applied Mathematics,
  296 (2016), pp.~116--126.

\bibitem{EK14}
{\sc A.~Eladdadi and P.~Kim}, {\em Mathematical Models of Tumor-Immune System
  Dynamics}, {S}pringer {P}roceedings in {M}athematics \& {S}tatistics,
  Springer, 2014.

\bibitem{FT07}
{\sc A.~B. Ferrari and E.~S. Titi}, {\em Gevrey regularity for nonlinear
  analytic parabolic equations}, Communications in Partial Differential
  Equations, 23 (1998), pp.~424--448.

\bibitem{Fried82}
{\sc A.~Friedman}, {\em Partial {D}ifferential {E}quations of {P}arabolic
  {T}ype}, Prentice-Hall, 1982.

\bibitem{Gali12}
{\sc G.~Galiano}, {\em On a cross-diffusion population model deduced from
  mutation and splitting of a single species}, Computers \& Mathematics with
  Applications, 64 (2012), pp.~1927--1936.

\bibitem{HR08}
{\sc M.~Hieber and J.~Rehberg}, {\em Quasilinear parabolic systems with mixed
  boundary conditions on nonsmooth domains}, SIAM Journal on Mathematical
  Analysis, 40 (2008), pp.~292--305.

\bibitem{JBAJ16}
{\sc R.~Jaroudi, G.~Baravdish, F.~{\AA}str\"om, and B.~T. Johansson}, {\em
  Source localization of reaction-diffusion models for brain tumors}, in
  Pattern Recognition, B.~Rosenhahn and B.~Andres, eds., vol.~9796 of {L}ecture
  {N}otes in {C}omputer {S}cience, Springer, Cham, 2016, pp.~414--425.

\bibitem{Kaba08}
{\sc S.~I. Kabanikhin}, {\em Definitions and examples of inverse and ill-posed
  problems}, Journal of Inverse and Ill-posed Problems, 16 (2008),
  pp.~317--357.

\bibitem{KLW17}
{\sc N.~I. Kavallaris, J.~Lankeit, and M.~Winkler}, {\em On a degenerate
  nonlocal parabolic problem describing infinite dimensional replicator
  dynamics}, SIAM Journal on Mathematical Analysis, 49 (2017), pp.~954--983.

\bibitem{Khoa2018}
{\sc V.~A. Khoa, T.~T. Hung, and D.~Lesnic}, {\em Uniqueness result for an
  age-dependent reaction-diffusion problem}, ArXiv e-prints,  (2018),
  \url{https://arxiv.org/abs/1806.03442}.

\bibitem{KTDT17}
{\sc V.~A. Khoa, M.~T.~N. Truong, N.~H.~M. Duy, and N.~H. Tuan}, {\em The
  {C}auchy problem of coupled elliptic sine-{G}ordon equations with noise:
  {A}nalysis of a general kernel-based regularization and reliable tools of
  computing}, Computers and Mathematics with Applications, 73 (2017),
  pp.~141--162.

\bibitem{KNT17}
{\sc M.~Kirane, E.~Nane, and N.~H. Tuan}, {\em On a backward problem for
  multidimensional {G}inzburg-{L}andau equation with random data}, Inverse
  Problems, 34 (2017).
\newblock 015008.

\bibitem{Kliba05}
{\sc M.~V. Klibanov}, {\em Estimates of initial conditions of parabolic
  equations and inequalities via lateral {C}auchy data}, Inverse Problems, 22
  (2006).

\bibitem{Klibanov2015}
{\sc M.~V. Klibanov}, {\em Carleman estimates for the regularization of
  ill-posed {C}auchy problems}, Applied Numerical Mathematics, 94 (2015),
  pp.~46--74.

\bibitem{Kli15}
{\sc M.~V. Klibanov}, {\em Carleman weight functions for solving ill-posed
  {C}auchy problems for quasilinear {P}{D}{E}s}, Inverse Problems, 31 (2015).
\newblock 125007.

\bibitem{KS91}
{\sc M.~V. Klibanov and F.~Santosa}, {\em A computational quasi-reversibility
  method for {C}auchy problems for {L}aplace`s equation}, SIAM Journal on
  Applied Mathematics, 51 (1991), pp.~1653--1675.

\bibitem{KP68}
{\sc R.~J. Knops and L.~E. Payne}, {\em On the stability of solutions of the
  {N}avier-{S}tokes equations backwards in time}, Archive for Rational
  Mechanics and Analysis, 29 (1968), pp.~331--335.

\bibitem{KWH16}
{\sc C.~K\"onig, F.~Werner, and T.~Hohage}, {\em Convergence rates for
  exponentially ill-posed inverse problems with impulsive noise:}, SIAM Journal
  on Numerical Analysis, 54 (2016), pp.~341--360.

\bibitem{KMU16}
{\sc T.~K\"uhn, S.~Mayer, and T.~Ulrich}, {\em Counting via entropy: {N}ew
  preasymptotics for the approximation numbers of {S}obolev embeddings}, SIAM
  Journal on Numerical Analysis, 54 (2016), pp.~3625--3647.

\bibitem{Lan16}
{\sc J.~Lankeit}, {\em Equilibration of unit mass solutions to a degenerate
  parabolic equation with a nonlocal gradient nonlinearity}, Nonlinear
  Analysis, 135 (2016), pp.~236--248.

\bibitem{LL67}
{\sc R.~Latt\`es and J.~L. Lions}, {\em M\'ethode de {Q}uasi-r\'eversibilit\'e
  et {A}pplications}, Dunod, Paris, 1967.

\bibitem{MOEP16}
{\sc K.~J. Mahasa, R.~Ouifki, A.~Eladdadi, and L.~de~Pillis}, {\em Mathematical
  model of tumor-immune surveillance}, Journal of Theoretical Biology, 404
  (2016), pp.~312--330.

\bibitem{Mey11}
{\sc M.~Meyries}, {\em Local well-posedness and instability of travelling waves
  in a chemotaxis model}, Advances in Differential Equations, 16 (2011),
  pp.~31--60.

\bibitem{MRS14}
{\sc M.~Meyries, J.~D.~M. Rademacher, and E.~Siero}, {\em Quasi-linear
  parabolic reaction-diffusion systems: {A} user{'s} guide to well-posedness,
  spectra, and stability of travelling waves}, SIAM Journal on Applied
  Dynamical Systems, 13 (2014), pp.~249--275.

\bibitem{Mill67}
{\sc K.~Miller}, {\em Stabilized quasi-reversibility and other
  nearly-best-possible methods for non-well-posed problems}, in Symposium on
  Non-Well-Posed Problems and Logarithmic Convexity, vol.~316 of {L}ecture
  {N}otes in {M}athematics, Springer, 1967, pp.~161--176.

\bibitem{Na17}
{\sc M.~Nakao}, {\em On the initial-boundary value problem for some quasilinear
  parabolic equations of divergence form}, Journal of Differential Equations,
  263 (2017), pp.~8565--8580.

\bibitem{Nam10}
{\sc P.~T. Nam}, {\em An approximate solution for nonlinear backward parabolic
  equations}, Journal of Mathematical Analysis and Applications, 367 (2010),
  pp.~337--349.

\bibitem{2010}
{\sc P.~T. Nam, D.~D. Trong, and N.~H. Tuan}, {\em The truncation method for a
  two-dimensional nonhomogeneous backward heat problem}, Applied Mathematics
  and Computation, 216 (2010), pp.~3423--3432.

\bibitem{Pao92}
{\sc C.~V. Pao}, {\em Nonlinear Parabolic and Elliptic Equations}, Springer,
  1992.

\bibitem{RE97}
{\sc T.~Regi\'nska and L.~Eld\'en}, {\em Solving the sideways heat equation by
  a wavelet-{G}alerkin method}, Inverse Problems, 13 (1997), pp.~1093--1106.

\bibitem{RR06}
{\sc T.~Regi\'nska and K.~Regi\'nski}, {\em Approximate solution of a {C}auchy
  problem for the {H}elmholtz equation}, Inverse Problems, 22 (2006),
  pp.~975--989.

\bibitem{RS92}
{\sc J.~Rubinstein and P.~Sternberg}, {\em Nonlocal reaction-diffusion
  equations and nucleation}, IMA Journal of Applied Mathematics, 48 (1992),
  pp.~249--264.

\bibitem{SS01}
{\sc B.~Sandstede and A.~Scheel}, {\em On the structure of spectra of modulated
  travelling waves}, Mathematische Nachrichten, 232 (2001), pp.~39--93.

\bibitem{Show83}
{\sc R.~E. Showalter}, {\em Cauchy {P}roblem for {H}yper-parabolic {P}artial
  {D}ifferential {E}quations}, {T}rends in the {T}heory and {P}ractice of
  {N}on-{L}inear {A}nalysis, Elsevier, 1983.

\bibitem{TAKL17}
{\sc N.~H. Tuan, V.~V. Au, V.~A. Khoa, and D.~Lesnic}, {\em Identifcation of
  the population density of a species model with nonlocal diffusion and
  nonlinear reaction}, Inverse Problems, 33 (2017).
\newblock 055019.

\bibitem{TDMK15}
{\sc N.~H. Tuan, B.~T. Duy, N.~D. Minh, and V.~A. Khoa}, {\em H\"older
  stability for a class of initial inverse nonlinear heat problem in multiple
  dimension}, Communications in Nonlinear Science and Numerical Simulation, 23
  (2015), pp.~89--114.

\bibitem{TQ11}
{\sc N.~H. Tuan and P.~H. Quan}, {\em Some extended results on a nonlinear
  ill-posed heat equation and remarks on a general case of nonlinear terms},
  Nonlinear Analysis: Real World Applications, 12 (2011), pp.~2973--2984.

\bibitem{Wei92}
{\sc D.~Wei}, {\em Existence, uniqueness, and numerical analysis of solutions
  of a quasilinear parabolic problem}, SIAM Journal on Numerical Analysis, 29
  (1992), pp.~484--497.

\bibitem{Zlat01}
{\sc M.~C. Zlatescu, A.~TehraniYazdi, H.~Sasaki, J.~F. Megyesi, R.~A. Betensky,
  D.~N. Louis, and J.~G. Cairncross}, {\em Tumor location and growth pattern
  correlate with genetic signature in oligodendroglial neoplasms}, Cancer
  Research, 61 (2001), pp.~6713--6715.

\end{thebibliography}
